\documentclass[11pt]{amsart}
\usepackage{amsmath,amssymb,latexsym,soul,cite,amsthm,color,enumitem,graphicx,tikz, mathtools,microtype}

\usepackage[pagebackref, colorlinks=true, urlcolor=blue, citecolor=red, linkcolor=blue]{hyperref}

\usepackage[english]{babel}
\usepackage[left=2.9cm,right=2.9cm,top=3cm, bottom=3cm]{geometry}

\numberwithin{equation}{section}

\newtheorem{theorem}{Theorem}[section]
\theoremstyle{plain}
\newtheorem{lemma}[theorem]{Lemma}
\theoremstyle{plain}
\newtheorem{proposition}[theorem]{Proposition}
\theoremstyle{plain}
\newtheorem{corollary}[theorem]{Corollary}

\theoremstyle{definition}
\newtheorem{remark}[theorem]{Remark}

\newcommand{\N}{{\mathbb N}}

\newcommand{\R}{{\mathbb R}}

\newcommand{\eps}{\varepsilon}
\newcommand{\e}{\epsilon}
\newcommand{\beq}{\begin{equation}}
\newcommand{\eeq}{\end{equation}}
\renewcommand{\le}{\leqslant}
\renewcommand{\ge}{\geqslant}

\newcommand{\cal}{\mathcal}

\def\Q{\mathcal{Q}}
\def\K{\mathcal{K}}

 \def\B{\mathcal{B}}

\def\T{\mathcal{T}}
\def\S{\mathcal{S}}

\title[Parabolic Harnack estimates for anisotropic slow diffusion]{Parabolic Harnack Estimates for anisotropic slow diffusion}
\author[S. Ciani]{Simone Ciani}
\address[S. Ciani]{Dipartimento di Matematica e Informatica ``Ulisse Dini'', Universit\`a degli Studi di Firenze, Viale Morgagni 67/A
50134, Firenze, Italy}
\email{simone.ciani@unifi.it}
\author[S. Mosconi]{Sunra Mosconi}
\address[S. Mosconi]{Dipartimento di Matematica e Informatica, Universit\`a degli Studi di Catania, Viale A. Doria 6, 95125 Catania, Italy}
\email{sunra.mosconi@unict.it}
\author[V. Vespri]{Vincenzo Vespri }
\address[V. Vespri]{Dipartimento di Matematica e Informatica ``Ulisse Dini'', Universit\`a degli Studi di Firenze, Viale Morgagni 67/A
50134, Firenze, Italy}
\email{vincenzo.vespri@unifi.it}

\begin{document}
%%%%%%%%%%%%%%%%%%%%%%%%%%%%%%%%%%%%%%%%%%%%%%%%%%%
%%%%%%%%%%%%%%%%%%%%%%%%%%%%%%%%%%%%%%%%%%%%
\begin{abstract} 
We prove a Harnack inequality for positive solutions of a parabolic equation with slow anisotropic spatial diffusion. After identifying its natural scalings, we reduce the problem to a Fokker-Planck equation and construct a self-similar Barenblatt solution. We exploit translation invariance to obtain positivity near the origin via a self-iteration method and deduce a sharp anisotropic expansion of positivity. This eventually yields a scale invariant Harnack inequality in an anisotropic geometry dictated by the speed of the diffusion coefficients.   As a corollary, we infer H\"older continuity, an elliptic Harnack inequality and a Liouville theorem. 


%%%%%%%%%%%%%%%%%%%%%%%%%%%%%%%%%%%%%%%%%%%%%%%%%%%

%%%%%%%%%%%%%%%%%%%%%%%%%%%%%%%%%%%%%%%%%%%%%%%%%%%

\end{abstract}

\maketitle

	\begin{center}
		\begin{minipage}{9cm}
			\small
			\tableofcontents
		\end{minipage}
	\end{center}

%%%%%%%%%%%%%%%%%%%%%%%%%%%%%%%%%%%%%%%%%%%%%%%%%%%%%%%%%%%%%%%%%%%%%%%
\section{Introduction}
We  are concerned  with solutions of the model parabolic anisotropic equation 
\begin{equation}
\label{mod}
\partial_{t}u=\sum_{i=1}^{N}\partial_{i}\big(|\partial_{i}u|^{p_{i}-2}\partial_{i}u\big)
\end{equation}
satisfied in a suitably weak sense in $\Omega\times (0, T)$, $\Omega\subseteq \R^{N}$ for powers $p_{i}>1$ for $i=1, \dots, N$. These kind of equations raised increasing interest in the last decades as they present an interesting feature, namely an {\em anisotropic diffusion} with orthotropic structure. Besides its inherent mathematical interest, the latter is useful when modelling diffusion in materials such as earth's crust or wood, where the velocity of propagation of diffusion varies according to the different orthogonal  directions. From the mathematical point of view, the principal part in \eqref{mod} arises as the Euler-Lagrange equation of a functional with {\em non-standard growth}, i.e. of the type
\[
\int F(\nabla u)\, dx,\quad \text{where}\quad  \frac{1}{C}\, (|z|^{p}-1)\le F(z)\le C\, (|z|^{q}+1)
\]
for some $p<q$, as opposed to the standard growth condition $p=q$. Starting from the pioneering examples in \cite{Marcellini-counter, Giaq}, it soon became apparent that the regularity theory for solutions of the corresponding Euler-Lagrange elliptic equation is much more delicate and rich than the standard one. Since then, the elliptic regularity theory  grew in considerable size. Even if this has not always been the case, the general principle underlying  to the theory is that most regularity results can be recovered when the power gap $q-p$ in the non-standard growth condition is small. Since  it would be impossible to collect here all the contributions, we refer to the surveys \cite[Section 6]{Mingione} and \cite{Marcellini} for a general overview of the subject and comprehensive bibliographic references. 

As the non-standard elliptic theory matured,  its parabolic counterpart became a research th\^eme as well. The delay in development was considerable, mainly because already the isotropic problem with $p_{i}\equiv p\ne 2$ presented great difficulties, solved (with respect to zero-th order regularity issues) in full generality only a decade ago  through the work of Di Benedetto and collaborators, see  \cite{DBGV-mono} and the literature therein. Nevertheless,  parabolic equations with non-standard growth were considered well before that (see e.g. \cite{Lions}), giving rise to a large amount of results on existence, well-posedness, $L^{\infty}$-estimates and diffusion analysis. For an extensive bibliography on this research, we refer to  \cite{AS} and for the theory of variational solutions to \cite[Section 12]{Marcellini} and the references therein.   

Despite some partial results, however, most of the regularity theory for the parabolic anisotropic equations is largely unknown. Up to our knowledge, local H\"older continuity of solution of \eqref{mod} was not known, as well as the validity of a suitable (necessarily intrinsic) parabolic Harnack inequality. The latter is precisely the aim of this paper, where we are going to prove the following result.

\begin{theorem} \label{Harnack-Inequality}
Let $u\ge 0$ be a  local weak solution to \eqref{mod} in $\Omega\times [-T,  T]$ and suppose  that
\begin{equation} \label{param}
    \forall i=1, \dots, N\quad   2<p_i < \bar{p}\, \bigg(1+\frac{1}{N} \bigg) \qquad \quad \bar{p} := \bigg(\frac{1}{N} \sum_{i=1}^N \frac{1}{p_i} \bigg)^{-1}<N
\end{equation}
and  $u(0,0)>0$. Then, there exist constants $C_{1}\ge 1, C_3\ge C_2\ge 1$ depending only on $N$ and the $p_{i}$'s such that, letting $M=u(0, 0)/C_1$ it holds 
\begin{equation}\label{Harnack}
  \frac{1}{C_{3}}\sup_{\K_{\rho}(M)}u(\,\cdot\, , - M^{2-\bar p}\, (C_{2}\, \rho)^{\bar p} )\le  u(0, 0) \le C_{3} \inf_{\K_{\rho}(M)} u(\,\cdot\, ,   M^{2-\bar{p}}\, (C_{2}\, \rho)^{\bar{p}})
    \end{equation}
    whenever $ M^{2-\bar p}\, (C_{3}\,  \rho)^{\bar p}<T$ and $\K_{C_{3}\, \rho}(M)\subseteq \Omega$, being
    \begin{equation}
    \label{intdef}
    \K_{r}(M):= \prod_{i=1}^N  \big{\{} |x_i| < M^{(p_i-\bar{p})/p_i} r^{\bar{p}/p_i}/2  \big{\}}.
    \end{equation}
\end{theorem}

Let us make some comments on the statement,  significance and proof of the previous theorem.

\smallskip
- {\bf The intrinsic geometry.} \ A parabolic Harnack inequality for a non-homogeneous equation such as \eqref{mod} cannot hold in classical form. This was first realised for the parabolic $p$-Laplacian equation
\begin{equation}
\label{modp}
\partial_{t}u=\Delta_{p}u
\end{equation}
through an analysis of the so-called Barenblatt fundamental solutions, a family of explicit solution encompassing most of the features which distinguish the classical heat equation (and its quasilinear non-degenerate counterpart) from \eqref{modp}. The correct formulation of the Harnack inequality for \eqref{modp} was first found in \cite{DB2} when $p\ge 2$, and it has an {\em intrinsic} form. To explain briefly this term let us focus on the $p\ge 2$ case of \eqref{modp}. A Harnack inequality for non-negative solutions of a parabolic equation expresses a point-wise control on the solution (e.g., a pointwise lower bound) in a full spatial neighbourhood of a point in terms of its value at that point. The parabolic nature of the equation allows for such a control to hold only after a positive time delay  (in the case of lower bounds) has passed. For the heat equation this {\em waiting time} only depends on the size of the region where we seek for the lower bound and not on the solution, while for the parabolic $p$-Laplacian equation \eqref{modp}, its length depends on the value of the solution at the chosen point: the word {\em intrinsic} refers (not only, but mainly) to this phenomenon. 

In the case of \eqref{modp}, the value of the solution at the chosen point only affects the waiting time, while for the anisotropic equation \eqref{mod}, it determines the full shape, or {\em geometry}, of the region where the control is available. This is seen in the definition \eqref{intdef} of the {\em intrinsic cubes}: indeed,  in $K_{r}(M)$, $r$ plays the r\^ole of an {\em anisotropic radius}, while $M$ prescribes  the {\em anisotropic  geometry}. To justify the first statement, notice that the Lebesgue measure of $K_{r}(M)$ is always $r^{N}$ , regardless of $M$. Regarding the second, one can follow the well-known principle that {\em higher exponents give slower diffusion}, so that lower values of $M\simeq u(0,0)$ squeeze $K_{r}(M)$ in directions of slower-than-average diffusion ($p_{i}-\bar p>0$) and stretch it in directions of faster-than-average  diffusion ($p_{i}-\bar p<0$). 

\smallskip
- {\bf Barenblatt solutions.} \  One of the main byproducts of our proof is the construction of a family of self-similar Barenblatt solutions for \eqref{mod} and the analysis on their  basic properties. Self-similar solutions are by now a classical th\^eme and have been extensively studied in various parabolic nonlinear frameworks, see e.g. \cite[Ch. 16]{VAZ} and the therein cited literature. Their r\^ole turned out to be pivotal in understanding the general behaviour of solutions and has often been an important stepping-stone for treating more general equations and formulating sensible statements on the general expected results: compare the classical works of Pini \cite{PIN} and Hadamard \cite{HAD}, later generalised in the linear measurable setting by Moser \cite{MOS} or, in the singular/degenerate case, the first works \cite{DB2}, \cite{KW} employing the Barenblatt solutions,  generalised  in \cite{DGV-acta, DGVsing}.

For equation \eqref{mod}, the explicit form of the Barenblatt solutions is however unknown at present, and their existence  is obtained through an abstract approach. Naturally, we cannot assume any a-priori regularity  and the method heavily relies on the identification of the natural scalings of \eqref{mod} mentioned above, allowing to formulate the right notion of self-similarity.  More details on the difficulties that this approach involves will be made in the comments to the proof below.

\smallskip
- {\bf Assumptions.} \ The main condition required in the Harnack inequality is \eqref{param}. On one hand, $p_{i}>2$ for all $i$ means that we are settling ourselves in the {\em slow diffusion regime}. The main feature of this framework is that, for example, solutions of \eqref{modp} for $p>2$ preserve compactness of the support forward in time  (as opposed to what happens for the heat equation).  In the setting of the anisotropic equation \eqref{mod}, the support moves in different directions with different speed, in a way which has been precisely quantified in \cite{Mosconi} and plays a r\^ole in our proof.  The other condition $p_{i}<\bar p\, (1+1/N)$ requires that the powers $p_{i}$ are not too sparse, following the above mentioned principle  in problems with non-standard growth. Local boundedness holds in the larger range $p_i< \bar p\, (1+2/N)$, but we are not aware of counterexamples if this condition is violated. It would be interesting to know wether the Harnack inequality holds true also for $p_{{\rm max}}\in [\bar p\, (1+1/N), \bar p \, (1+2/N))$ but, if so,  its proof likely requires different techniques than the ones employed here. 

A few comments on the constants $C_{i}$ in the statement. As mentioned above, the Barenblatt solution we use is constructed in an abstract way and we do not know if a uniqueness theorem (up to translation and scaling) holds. The constant depends on a lower bound on the Barenblatt solution, hence, ultimately on the choice of the latter. Thus, it is rather undetermined from the  quantitative point of view.

Finally, the number $u(0, 0)$ is not a-priori well-defined for a weak solution. However, thanks to \cite[Corollary 4.3]{Mosconi}, any solution of \eqref{mod} under assumptions  \eqref{param} possesses an essentially u.s.c. representative, allowing to give a meaning to $u(0, 0)$. This ambiguity in the choice will then be eliminated by the a-posteriori continuity of the solution. Clearly, the theorem is meaningful only when $u(0, 0)>0$, for otherwise the claimed bounds trivially holds (assuming $\inf \emptyset=+\infty$, $\sup \emptyset=-\infty$).

\smallskip
- {\bf Outline of the proof}. \  As already mentioned, our first task is to build a family of Barenblatt solutions. We find all the natural scalings of \eqref{mod} and construct a bijection between solutions of \eqref{mod} and solutions of an anisotropic Fokker-Plank equation (see e.g. \cite{Carrillo-Toscani} for a similar approach). We then seek for a  stationary solution of the latter, which is found through a fixed point argument and comparison principles. Here, the slow diffusion regime plays a pivotal r\^ole in recovering sufficient compactness to apply Shauder fixed point theorem. Let us note that we rely on a weak continuity result (Lemma \ref{lemma31}, point 3) of independent interest, which we were not able to find in the literature. 

At this stage, the stationary solution of the Fokker-Plank equation is a rather irregular object  of little use. However, exploiting its correspondence to a Barenblatt self-similar solution of \eqref{mod} and using a self-iteration method based on comparison principles and translation invariance,  we are able to prove a positive lower bound in a small neighbourhood of the origin. Transferring the bound to the Barenblatt solution, we find a quantitative expansion of positivity rate for it. 

We then proceed in a manner reminiscent of the proof in \cite{DB2} of the Harnack inequality for \eqref{modp}, namely finding a positivity set  and then expand it forward in time through comparison with Barenblatt solutions. For the first step, we actually employ a simplification described in \cite{DGV-acta}, which makes use of the so-called {\em Clustering lemma} of \cite{localclustering}. We have to face two main difficulties: the intrinsic geometry of the problem, contrary to what happens in most instances of the theory, involves not only the time variables but also, and mainly, the spatial ones (in an anisotropic way). Secondly, 
even disregarding the geometry, the natural intrinsic cubes as per \eqref{intdef} come from a quasi-metric rather than from a metric. To face the first difficulty we heavily rely on the natural transformations leaving \eqref{mod} invariant; for the second one, we prove a general abstract version of the so-called Krylov-Safonov trick, of independent interest (Lemma \ref{Krylov-Safonov}).

\smallskip
- {\bf Consequences of Theorem \ref{Harnack-Inequality}}.\  A first corollary of the Harnack inequality is the H\"older continuity of solutions of \eqref{mod}, whose detailed proof is described in \cite{CianiVespri}. However, much more regularity is to be expected, as suggested by the ellitpic case briefly discussed below. 

An intrinsic Harnack inequality immediately follows from Theorem \ref{Harnack-Inequality} for solutions of 
\begin{equation}
\label{mod2}
\sum_{i=1}^{N}\partial_{i}(|\partial_{i}u|^{p_{i}-2}\partial_{i}u)=0.
\end{equation} 
Even in the elliptic case,  homogeneity is still missing, suggesting that  any scale invariant Harnack inequality must be of intrinsic form, as in the parabolic case. We state our Harnack inequality for \eqref{mod} in the following corollary.

\begin{corollary} Let $\K_{r}(M)$ be as in \eqref{intdef}  and $p_{i}$ as in \eqref{param}.  There exist constants $C_{1}, C_{2}\ge 1$, depending on $N$ and the $p_{i}$'s such that if $u\ge 0$ weakly solves \eqref{mod2} in $\K_{C_{2} \rho}(M)$, where $M=u(0)/C_{1}>0$, then
\begin{equation}
\label{H}
C_{2}^{-1} \sup_{\K_{\rho}(M)}u\le u(0)\le C_{2} \inf_{\K_{\rho}(M)}u.
\end{equation}
\end{corollary}
Notice that condition \eqref{param} on the powers $p_{i}$ is in fact of parabolic nature, tied to the proof of \cite{Mosconi}. The proof of the elliptic Harnack inequality under the more natural condition $p_{{\rm max}}< N\, \bar p /(N-\bar p)$ is the object of future work.  
The scale invariance of the Harnack inequality, i.e., the fact that the constants in \eqref{H} do not depend on the radius or the solution, is crucial when dealing with Liouville-type theorems like the one below, proved in a standard way in the last section.

\begin{corollary} \label{liouville} Under assumption  \eqref{param}, any weak solution of \eqref{mod2} in the whole $\R^{N}$ bounded from below is constant.
\end{corollary}

- {\bf Comparison with previous results}.\ 
Local boundedness of the solutions of \eqref{mod} has been first proved in \cite{Mingqi} under the condition $ p_{\rm max}<\bar p\, (1+2/N)$.  Some early regularity results in the plane are considered in \cite{MK}, and regularity for parabolic problems with non-standard growth of $p(x)$ type are contained in \cite{AMS, AZ, XC}. The $p(x)$ growth condition does not cover the simple equation \eqref{mod} and we are not aware of proofs of the H\"older continuity of solutions of the latter in general dimensions (see \cite[Remark 1.4]{BB} for a discussion of previous attempts), let alone of the Harnack inequality. 

In the elliptic setting much more is known regarding the regularity of solutions of \eqref{mod2}, or for more general non-standard equations, see \cite[Sections 5 and 6]{Marcellini} for the relevant literature. The most up-to-date result for \eqref{mod2} is in \cite{BB}, where  the Lipschitz regularity of its {\em bounded} solutions is proved {\em for any choice} of $p_{i}\ge 2$. 
The Harnack inequality for non-standard elliptic problems has been the object of various works: \cite{Alk, BCM, Ok, LS, HKLMP, HKL, Toivanen} focus on isotropic equations with non-standard growth of $p(x)$-type, while \cite{MP, L} deal with energies with Uhlenbeck structure and non-standard growth. However, none of the frameworks considered therein cover the anisotropic equation \eqref{mod2}: indeed, its Euler-Lagrange equation is degenerate/singular on the union of the coordinate axes, while non-standard functionals of $p(x)$- or Uhlenbeck-type exhibit this problem only at the origin. Moreover, as already remarked, the relevant feature of \eqref{H} lies in its scale invariance, while (quite naturally) this is not to be expected for the problems considered in the cited works, where either the constant depends on $u$ and $r$ or there is an additional term of non-homogeneous type. 
 
 \smallskip
- {\bf Structure of the paper}.\ 
Section \ref{pre} collects preliminary results, most of which are modifications of well-known theorems. The most relevant part is subsection \ref{scaling}, where we set up the geometry related to the natural scaling of the equation. In Section \ref{bare} we build the Barenblatt solution and study its positivity set. Section \ref{final} contains the proof of the main theorem, splitted in several lemmas.

\smallskip
- {\bf Notations}: 
\begin{itemize}
\item[-]
For $\xi\in \R$ and $p>2$ we  let $(\xi)^{p-1}=|\xi|^{p-2}\xi$. 
\item[-]
If $E$ is a measurable subset of $\R^{N}$, we denote by $|E|$ its Lebesgue measure. 
\item[-]
For $r>0$, $K_{r}(\bar x)$  stands for the cube $\{|x_{i}-\bar x_{i}|\le r/2: i=1, \dots, N\}$ and we write $K_{r}=K_{r}(0)$; the standard cylinder is denoted by $Q_{r}^{-}=K_{r}\times (-r^{2}, 0]$.  Notice that $|K_{1}|=|Q_{1}^{-}|=1$
\item[-]
Given $T\in (0, +\infty]$ and $\Omega \subset \R^N$ a rectangular domain, we let $\Omega_T= \Omega\times (0, T)$ while $S_T$ denotes the stripe $S_T=\R^N\times (0,T) $; more generally, for $s<t$ we will set $\Omega_{s,t}=\Omega\times (s,t)$ and $S_{s, t}= \R^{N}\times (s, t)$.   
\item[-]
For a measurable $u$, by $\inf u$ and $\sup u$ we understand the essential infimum and supremum, respectively; when  $u$ is defined on all of $\R^{N}$, we  let $\|u\|_{p}=\|u\|_{L^{p}(\R^{N})}$ for $1\le p\le \infty$; when $u:E\to \R$  and $a\in \R$, we  omit the domain when considering sub/super level sets, letting $\big[u\gtreqless a\big]=\big\{x\in E: u(x)\gtreqless a\big\}$; if $u$ is defined on some $\Omega_{T}$, we let  $u_{t}(x)=u(x, t)$ while $\partial_iu=\frac{\partial }{\partial x_i}u$, $\partial_tu=\frac{\partial }{\partial t}u$ denote the distributional derivatives.
\end{itemize}

\section{Preliminaries}\label{pre}
In this section we will discuss the functional analytic setting we pose ourselves in, the scaling properties of the solutions of \eqref{mod}, some basic energy estimates and the resulting anisotropic De Giorgi type lemma, comparison principles for \eqref{mod} and the associated Fokker-Planck equation and solvability of the Cauchy problem for \eqref{mod} for suitable initial data. Most of the material is standard, except maybe the discussion  in section \ref{scaling}. 

\subsection{Functional setting.}
Given ${\bf p}=(p_{1}, \dots, p_{N})$, with $p_{i}>1$, $i=1, \dots, N$ and $\Omega$ a rectangular domain, we define
\[ W^{1,{\bf{p}}}_o(\Omega):= \{ v \in W^{1,1}_o(\Omega) |\,  \partial_i v \in L^{p_i}(\Omega) \},\]
% \[ W^{1,{\bf{p}}}(\Omega):= \{ v \in L^1(\Omega) |\,  \partial_i v \in L^{p_i}(\Omega) \}, \]
\[  W^{1,{\bf{p}}}_{loc}(\Omega):= \{ v \in L^1_{loc}(\Omega) |\,  \partial_i v \in L^{p_i}_{loc}(\Omega) \}, \]
and for $s<t$
\[ L^{{\bf{p}}}(s, t;W^{1,{\bf{p}}}(\Omega)):= \{v \in L^1(s, t;W^{1,1}(\Omega))|\, \partial_i v, v \in L^{p_i}(\Omega_{s,t})   \}, \]
\[L^{{\bf{p}}}_{loc}(s, t;W^{1,{\bf{p}}}_{loc}(\Omega)):= \{v \in L^1_{loc}(s, t;W^{1,1}_{loc}(\Omega))|\, \partial_i v \in L^{p_i}_{loc}(\Omega_{s,t})   \}, \] 
\[L^{{\bf{p}}}_{loc}(s, t;W^{1,{\bf{p}}}_o(\Omega)):= \{v \in L^1_{loc}(s, t;W^{1,1}_o(\Omega))|\, \partial_i v \in L^{p_i}_{loc}(\Omega_{s,t})   \}. \] 
\noindent \vskip0.3cm 

\noindent A function \[ u \in L^{\infty}_{{\rm loc}}(s, t; L^2_{loc}(\Omega)) \cap L^{\bf{p}}_{loc}(s, t;W_{loc}^{1,{\bf{p}}}(\Omega))\] is  a local weak solution of \eqref{mod} in $(s, t)\times \Omega$, if for almost every $s<t_1<t_2<t$ and any  $\varphi \in C^{\infty}_{loc}(s,t;C_o^{\infty}(\Omega))$ it holds
\[
\int_{\Omega} u_{t_{1}}\,  \varphi_{t_{1}} \, dx -\int_{\Omega} u_{t_{2}}\,  \varphi_{t_{2}} \, dx+ \int_{t_1}^{t_2} \int_{\Omega} (-u \, \partial_t \varphi + \sum_{i=1}^N \, (\partial_i u)^{p_i-1} \, \partial_i \varphi) \, dx\, dt=0.
\]
By an approximation argument the latter actually holds for any test function $\varphi \in W^{1,2}_{loc}(s,t;L^2_{loc}(R))\cap L^{\bf{p}}_{loc}(0,T;W^{1,{\bf{p}}}_o(R))$ for any rectangular domain $R\subset \subset \Omega$. By a local weak solution of \eqref{mod} in $ S_{\infty}$, we mean a function $ u \in L^{\infty}(\R_+; L^2_{loc}(\R^N)) \cap L^{\bf{p}}_{loc}(\R_+;W^{1,{\bf{p}}}_{loc}(\R^N))$ such that for each $T>0$ and $\Omega \subset \subset \R^N$, $u$ is a weak local solution in $\Omega_T$. Finally, by an {\it $L^{\bf{p}}$ solution of \eqref{mod}} in $S_{T}$, we mean a local weak solution $u$ in $S_{T}$ such that $u \in \cap_{i=1}^N L^{p_i}(S_T)$.\vskip0.2cm
\noindent Next we recall the following  anisotropic embedding, obtained, e.g., from \cite[Theorem 2.4]{Mosconi} with $\sigma=2$, $\alpha_{i}\equiv 1$, $\theta=\bar p/\bar p^{*}$ and the generalised Young inequality. 
\begin{lemma}[Parabolic anisotropic Sobolev embedding]
Let $\Omega\subseteq \R^{N}$ be a bounded rectangular domain. There exists a constant $C=C(N,{\bf p}) <+\infty$ such that for any $u\in L^{1}(0, T; W^{1, 1}_{o}(\Omega))$ it holds
\begin{equation} \label{embedding}
    \ \int_{\Omega_{T}} |u|^l \, dx\, dt \le C\,  \bigg( \sup_{t\in [0,T]} \int_{\Omega} |u|^2(x, t)\,  dx + \int_{\Omega_{T}} \sum_{i=1}^N |\partial_i u|^{p_i} \, dx\, dt  \bigg)^{(N+{\bar{p}})/N}
\end{equation} whenever 
\begin{equation} \label{embedding-exponent}
\frac{2N}{N+2}\le \bar p:=\big(\frac{1}{N}\sum_{i=1}^{N}\frac{1}{p_{i}}\big)^{-1}<N,\qquad  l:={\bar{p}}\,  \big(1+ \frac{2}{N}  \big).
\end{equation}
\end{lemma}
\noindent
By applying the so-called {\it Local Clustering} lemma in \cite{localclustering} to $\min\{u, 1\}+\delta$ and letting $\delta\downarrow 0$, we get the following alternative form, which will be used in the following.
\begin{lemma}[Local clustering] \label{localclustering-revisited}
Let $u \in W^{1,1}(K_{\rho})$ satisfy for some constants $\bar C>0, \bar \alpha \in (0,1)$ 
\[
    \int_{K_{\rho}} |D\, (u-1)_{-}| \, dx \le \bar C\, \rho^{N-1} \qquad \text{\&} \qquad |[u>1] \cap K_{\rho}| \ge \bar\alpha\,  |K_{\rho}|.
\]
Then for every $ \lambda, \nu \in (0,1)$ there exists  $ y \in K_{\rho}$ and a number $\eps= \eps(\lambda,\nu,\bar C,\bar \alpha,N)\in (0,1)$ such that $y+K_{\eps\rho}\subseteq K_{\rho}$ and
\[
 |[u\ge  \lambda] \cap (y+K_{\eps \rho})| > (1-\nu) |K_{\eps \rho}|.
\]
Moreover, $\eps$ can be chosen arbitrarily small.
\end{lemma}

 \subsection{Scaling properties} \label{scaling}

We omit the proof of the following proposition, which is just a direct computation.
\begin{proposition} \label{transformation-group}
Let $u$ weakly solve \eqref{mod} in $\Omega_{T}$. For  $\theta,\rho>0$ define
\[
T_{\rho,\theta} (y,s)= \big(\theta^{(p_i-\bar{p})/p_i}\,  \rho^{\bar{p}/p_i}\,  y_i , \theta^{2-\bar{p}}\,  \rho^{\bar{p}}\,  s\big).
\]
Then the transformed function
\begin{equation}
\label{transformation}
    \T_{\rho, \theta}u\, (y,s)= \frac{1}{\theta}\, u \big( T_{\rho,\theta} (y,s) \big)
\end{equation}
weakly solves \eqref{mod} in $T_{\rho, \theta}^{-1}(\Omega_{T})$.
\end{proposition}
Due to the latter proposition, it is convenient to set
\begin{equation}\label{alpha}
\sigma:= N\, (\bar{p}-2)+\bar{p}\qquad  \alpha:= \frac{N}{\sigma}, \qquad \text{and} \qquad\alpha_i:=\frac{N\, (\bar p -p_{i})+\bar p}{\sigma \, p_{i}},  
\end{equation}
(notice that, under assumption \eqref{param}, $\alpha_{i}>0$ for all $i=1, \dots, N$),  so that
\[
\T_{\rho, \theta\rho^{-N}} u\, (y, s)= \frac{\rho^{\sigma\, \alpha}}{\theta}\, u \big(\theta^{(p_i-\bar{p})/p_i} \, \rho^{\sigma\, \alpha_{i}} \, y_i , \theta^{2-\bar{p}} \, \rho^{\sigma}\, s\big).
\]
The following properties of $\T$ will be useful throughout calculations:
\[
 \T_{\rho_{1},\theta_{1}} \circ \T_{\rho_{2},\theta_{2}}= \T_{\rho_{1} \rho_{2}, \theta_{1} \theta_{2}}\qquad \T_{\rho,\theta}^{-1}= \T_{\rho^{-1},\theta^{-1}}
\]
and similarly for the trasformation $T_{\rho, \theta}$.
The previous scaling suggests the natural geometry where to settle problem \eqref{mod}. More precisely, we define the intrinsic anisotropic cube as
\[
    \K_{\rho}(\theta):=  T_{\rho, \theta}(K_{1}),\qquad \K_{\rho}:=\K_{\rho}(1),
\]
(here and in what follows we will use the same symbol $T_{\rho, \theta}$ to denote the action of $T_{\rho, \theta}$ on the space variables only) and the intrinsic anisotropic cylinders as
\[
    \Q_{\rho}^{-}(\theta):= T_{\rho, \theta}(Q_{1}^{-}), \qquad \Q_{\rho}^{-}:=\Q_{\rho}^{-}(1).
\]
Notice that in the anisotropic cubes the parameter $\rho$ prescribes the size, while $\theta$ determines its anisotropic geometry: indeed, the volume of $\K_{\rho}(\theta)$ does not depend on $\theta$, since for each $\theta,\rho>0$ 
\[
|\K_{\rho}(\theta)|= \rho^N .
\] 
The following property can be readily checked: 
\begin{equation}
\label{tr}
T_{\rho, \theta}(K_{R})=\K_{R\,  \rho }(R\, \theta), \qquad T_{\rho, \theta}(Q_{R}^{-})=\Q_{R\, \rho}^{-}(R\, \theta),
\end{equation}
and in particular it holds $\K_{R}(R)=K_{R}$ and $\Q_{R}^{-}(R)=Q_{R}^{-}$.
 An important special case of the  transformation \eqref{transformation} is obtained when $v$ does not depend on the time variable and $\theta=\rho^{-N}$: using the notations in \eqref{alpha} we define
\begin{equation}
\label{trho}
\T_{\rho} v\, (y):= \T_{\rho, \rho^{-N}} v\, (y)=\rho^{\sigma\alpha}\,  v\big( \rho^{\sigma\alpha_{i}}\, y_i\big).
\end{equation}
By a change of variables one can  check that $\T_{\rho}:L^{1}(\R^{N})\to L^{1}(\R^{N})$ is an isometry, and moreover
\begin{equation}
\label{taus}
(\T_{\rho,\rho^{-N}}u)_{s}=\T_{\rho} u_{\rho^{\sigma}s}\,.
\end{equation}
As we will see, reasonable solutions of \eqref{mod} preserve the $L^{1}$-norm in time, and therefore we will say that $u$ is a {\it self-similar solution} of \eqref{mod} in $S_{\infty}$ if $\T_{\rho, \rho^{-N}} u=u$ for all $\rho>0$.

\noindent Closely related transformations are 
\begin{equation}\label{continuous-transformation}
   \Phi u \, (y, s):= e^{\alpha s} \, u( e^{\alpha_i s}\, y_i, e^s), \quad \text{and the inverse} \quad \Psi w\, (x, t)=t^{-\alpha}\,  w(t^{-\alpha_i}\, x_i, \log t).
\end{equation}
Clearly if $u$ is defined on $\R^{N}\times [t_{0}, +\infty)$, $t_{0}>0$ then $\Phi u$ is defined on $\R^{N}\times [\log t_{0}, +\infty)$ and vice-versa if $w$ is defined on $\R^{N}\times [s_{0}, +\infty)$, $\Psi w$ is defined on $\R^{N}\times [e^{s_{0}}, +\infty)$. Due to \eqref{taus}, it holds
\begin{equation}
\label{phit}
(\Phi u)_{s}=\T_{e^{s/\sigma}}u_{e^{s}},\qquad (\Psi w)_{t}=\T_{t^{1/\sigma}}w_{\log t}.
\end{equation}
\noindent 
By \cite{CianiVespri2}, $\Phi$ brings solutions of  \eqref{mod} in $S_{t_{0}, \infty}$, to solutions of the anisotropic Fokker-Planck  equation 
\begin{equation}\label{FokkerPlanck}
  \partial_{s}w= \sum_{i=1}^{N}\partial_i \big[( \partial_i w  )^{p_i-1} +\alpha_i y_i w \big] 
 \end{equation}
in $S_{\log t_{0}, \infty}$ and $\Psi$ does the opposite. Using \eqref{taus}, \eqref{phit} together with $\T_{\rho_{1}}\T_{\rho_{2}}=\T_{\rho_{1}\rho_{2}}$, we see that for a solution $u$ of \eqref{mod} in $S_{\infty}$ 
\[
(\Phi \T_{\rho, \rho^{-N}}u)_{s}=\T_{e^{s/\sigma}}\big(\T_{\rho, \rho^{-N}} u\big)_{e^{s}}=\T_{e^{s/\sigma}}\T_{\rho} u_{\rho^{\sigma}e^{s}}=\T_{\rho e^{s/\sigma}}u_{\rho^{\sigma}e^{s}}=(\Phi u )_{\sigma\log \rho+s}
\]
for every $\rho>0$, from which we readily infer the following proposition.
\begin{proposition}
A function $u$ is a self-similar solution of \eqref{mod} in $S_{\infty}$ if and only if $\Phi u$ is a stationary solution of \eqref{FokkerPlanck} in $\R^{N+1}$. 
\end{proposition}

 \noindent In the following we will call a self-similar solution to \eqref{mod} in $S_{\infty}$ a {\it{Barenblatt fundamental solution}}, and we will denote it with $\B$, in analogy with the literature about the $p$-Laplacian. Moreover, we will henceforth use coordinates $(x, t)$ for the prototype equation \eqref{mod} and $(y, s)$ for the Fokker-Planck equation \eqref{FokkerPlanck}.

\subsection{Energy inequality and consequences}
The following energy estimate for solutions of  \eqref{mod}, is well known, see e.g. \cite[Lemma 3.1]{Mosconi} for a proof.
\begin{lemma}[Energy inequality] 
Let $u$ be a local weak solution of \eqref{mod} in $K_{R}\times [s_{1}, s_{2}]$. Then,
for each test function of the form 
\[ 
\eta(x, t) = \prod_{i=1}^N \eta_i^{p_i}(x_i,t), \qquad
 \eta_i \in C^{\infty}(s_{1},  s_{2};C^{\infty}_{c}(-R/2, R/2)) 
\] 
we have, for some $C=C(N, {\bf p})>0$,
\begin{equation} \label{energy-estimates}
\begin{aligned}
    \int_{K_R}& (u_{t}-k)_{\pm}^2\,  \eta_{t} dx \bigg|_{t=s_{1}}^{t=s_{2}} + \frac{1}{C}\sum_{i=1}^N \int_{s_{1}}^{s_{2}}\int_{K_R} |\partial_i (\eta\, (u-k)_{\pm})|^{p_i}  \, dx\, d\tau \\
    & \le  C \bigg{\{} \int_{s_{1}}^{s_{2}}\int_{K_{R}} (u-k)_{\pm}^2\,  |\partial_t \eta| \, dx\,d\tau + C\sum_{i=1}^N \int_{s_{1}}^{s_{2}}\int_{K_{R}} (u-k)_{\pm}^{p_i}\,  |\partial_i \eta^{\frac{1}{p_{i}}}|^{p_i} \, dx\, d\tau  \bigg{\}} .
    \end{aligned}
\end{equation} \noindent
\end{lemma}

In a standard way we can prove a de Giorgi-type Lemma. 
\begin{lemma}[De Giorgi  Lemma]\label{DG} Let $u\ge 0$ be a local weak solution to \eqref{mod} in $Q_{1}^{-}$ and ${\bf p}$ obey \eqref{embedding-exponent}. Then for every $a\in(0,1]$ there exist  $\mu_{a}>0$ depending on $a, {\bf p}$ and $N$  such that
\[
        |[u \le a]\cap Q_1^{-}| \le \mu_{a}\,  |Q_1^{-}| \qquad \Rightarrow\qquad
   \inf_{Q_{1/2}^{-}} u \ge \frac{1}{2}\, a.
\]
If, in addition, it holds $u\le 1$ in $Q_{1}^{-}$, 
\[
        |[u\ge a]\cap Q_1^{-}| \le \mu_{a}\,  |Q_1^{-}| \qquad \Rightarrow \qquad 
    \sup_{Q_{1/2}^{-}}u \le \frac{3}{2}\, a.
\]
\end{lemma} 
\begin{proof}
We give a brief proof of the second statement, the first one being analogous. Let, for $n \in \N$,
\[
    \rho_n= \bigg(\frac{1}{2}+\frac{1}{2^{n+1}} \bigg),\quad \quad 
    k_n=  a\bigg(\frac{3}{2}-\frac{1}{2^{n+1}} \bigg), \quad \quad K_{n}=K_{\rho_{n}}\quad \quad Q_n^{-}=Q_{\rho_{n}}^{-}.
\]
We apply \eqref{energy-estimates} to $(u-k_{n})_{+}$ with  $\eta_n$ of the stipulated form with $\eta_{n}=1$ in $Q_{n+1}$, $\eta_{n}=0$ on the parabolic boundary of $Q_{n}^{-}$ and 
\[
0\le \eta_{n}\le 1,\qquad |\partial_{t}\eta_{n}|+|\nabla \eta_{n}|\le C\, 2^{n}.
\] 
Since $\eta_{n}(\cdot, -1)\equiv 0$, the energy inequality \eqref{energy-estimates} together with the bound $|u|\le 1$ yields
\begin{equation}
\label{ref}
    \begin{split}
     \int_{K_{n}}& (u_{t}-k_n)_+^2 \, (\eta_{n})_{t} \, dx + \frac{1}{C}\sum_{i=1}^N \int_{Q_n^{-}} |\partial_i(\eta_{n} \, (u-k_n)_+)|^{p_i} \, dx\, d\tau\\
     &\le   C\, 2^{c\, n}\,  \bigg{\{}  \int_{Q_n^{-}} (u-k_n)_+^2  \, dxdt + \sum_{i=1}^N  \int_{Q_n^{-}} |(u-k_n)_+|^{p_i}  \, dx\, d\tau  \bigg{\}} \le C\,  2^{c \, n} |[u>k_n] \cap Q_n^{-}|
    \end{split}
\end{equation}
for all $t\in [-\rho_{n}^2, 0]$, where $C=C(N, {\bf p}, a)$ and $c=c({\bf p})$.
Let $A_n=[u>k_n] \cap Q_n^{-}$. By Chebyshev's inequality and the assumptions on $\eta_{n}$ it holds, for $l$ given in \eqref{embedding-exponent}
\[
\big( \frac{a}{2^{n+1}}  \big)^l  |A_{n+1}|=(k_{n}-k_{n+1})^{l}\, |A_{n+1}| \le \int_{Q_{n+1}^{-}} |(u-k_{n})_+|^{l} \, dx\, d\tau \le  \int_{Q_{n}^{-}} |(u-k_{n})_+\, \eta_{n}|^{l}\, dx\, d\tau
\]
and chaining Sobolev's embedding  \eqref{embedding} on the right and \eqref{ref} (notice that $\eta_{n}^{2}\le \eta_{n}$), we get 

\[
\begin{split} 
\big( \frac{a}{2^{n+1}}  \big)^l  |A_{n+1}|& \le C\,  \bigg( \sup_{t\in (- \rho_n^{2},0]} \int_{K_n}  (u_{t}-k_n)_+^{2}\, (\eta_{n})_{t} dx + \sum_{i=1}^N  \int_{Q_n^{-}} |\partial_i (\eta_{n}\, (u-k_n)_+)|^{p_i} \, dx\, d\tau    \bigg)^{\frac{N+\bar{p}}{N}} \\
& \le C 2^{c\, n}  \,     |A_n|^{1+\bar p/N}, 
\end{split}
\]
for some bigger $C$, $c$. By the fast convergence Lemma \cite[Lemma 5.1 chap. 2]{DBGV-mono},  if $|A_{0}|$ is sufficiently small (depending on $N, {\bf p}$ and  $a$), $|A_{n}|\to 0$ for $n\to \infty$, implying the claim.
\end{proof} \noindent 

\begin{remark}\label{DGc}
Applying the transformation $\T_{\rho, \theta}$ and recalling \eqref{tr}, we get a similar statement for any solution in $\Q_{\rho}(\theta)^{-}$. For example, if $u\ge 0$ solves \eqref{mod} in  $\Q_{\rho}^{-}(\theta)$, there exists $\mu_{1}=\mu_{1}( N, {\bf p})>0$ such that 
\[
|[u \le \theta ] \cap \Q_{\rho}^{-}(\theta)| \le \mu_{1}\, |\Q_{\rho}^{-}(\theta)|\quad \Rightarrow\quad  \inf_{\Q_{\rho/2}^{-}(\theta/2)}u\ge \theta/2.
\] 
\end{remark} 

\subsection{Comparison Principles}
We consider in this section the Cauchy problem for \eqref{mod}, namely
\begin{equation}
\label{cp}
\begin{cases}
\partial_{t}u=\sum_{i=1}^{N}\partial_{i}((\partial_{i} u)^{p_{i}-1})&\text{weakly in $\Omega_{T}$}\\
u_{t}\to u_{0}&\text{strongly in $L^{2}(\Omega)$}
\end{cases}
\end{equation}
and a similar one for the Fokker-Planck equation \eqref{FokkerPlanck}.
Given two solutions $u, v$ of this problem, we say that $u\ge v$ on the parabolic boundary of $\Omega_{T}$ if $(u-v)_{-}\in L^{{\bf p}}(0, T; W^{1, {\bf p}}_{0}(\Omega))$ and $u_{0}\ge v_{0}$.
From the monotonicity of the principal part in \eqref{mod} we have the following classical comparison principle.
\begin{proposition}(Local comparison principle) 
Let $\Omega$ be a bounded rectangular domain, $u, v$ be weak solutions of  \eqref{cp} in $\Omega_T$. If $u \ge v$ on the parabolic boundary  of $\Omega_T$, then $u \ge v$ in $\Omega_T$.
\end{proposition} \noindent Next we provide a comparison principle for the class of $L^{{\bf p}}$-solutions, that will prove to be useful for next purposes. We sketch the proof inasmuch the r\^ole of greater integrability can be fully exploited. 
\begin{proposition} \label{com}
Let $u,v $ be two $L^{{\bf{p}}}$ solutions of \eqref{cp} in $S_T$, satisfying $u_{0} \ge v_{0}$ for  $u_{0}, v_{0} \in L^{2}(\R^{N})$. Then $u \ge v$ in $S_T$. 
\end{proposition}
\begin{proof} First notice that if $u$ is an $L^{{\bf{p}}}$ solution of \eqref{cp} in $S_T$ with $u_0 \in L^2(\R^N)$, then $u \in L^{\bf{p}}(0,T;W^{1,\bf{p}}(\Omega))$. Indeed, by the energy estimate \eqref{energy-estimates} with a standard cut-off, we deduce that 
\[
\sum_{i=1}^{N}\|\partial_{i}u\|_{L^{p_{i}}(S_{T})}\le C\, \big(\|u_{0}\|_{2}^{2}+\sum_{i=1}^{N}\|u\|_{L^{p_{i}}(S_{T})}\big),
\]
and similarly for $v$.  Secondly, we test the equations for $u$ and $v$ with $ (v-u)_+ \zeta$, where $\zeta$ a cut-off function between the balls $B_R$ and $B_{2R}$, independent of time and such that $|\partial_i \zeta| \le C/R$, $0 \le \zeta \le 1$. Subtracting the resulting integral equalities and using $u_{0}\ge v_{0}$  we have, for every $t>0$,
\[
    \begin{split}
\int_{B_{R} \cap [v \ge u]} &\zeta \, (v-u)^2(x,t)\, dx + \sum_{i=1}^N \int_0^t \int_{B_{2R} \cap [v\ge u]}  \zeta\, \big( (\partial_i v)^{p_i-1}-(\partial_i u)^{p_i-1}\big) \, \partial_i  (v-u)\,  dx\, d\tau\\
& = -\sum_{i=1}^N \int_0^t \int_{B_{2R} \cap [v\ge u]}  (v-u)\,  \big( (\partial_i v)^{p_i-1}-(\partial_i u)^{p_i-1}\big) \, \partial_i  \zeta\,  dx\, d\tau \\
    &\le \sum_{i=1}^N \frac{C}{R} \int_0^t \int_{B_{2R} \cap [v \ge u]} ( |\partial_i v|^{p_i-1}\, |u| +|\partial_i v|^{p_i-1} \, |v|+|\partial_i u|^{p_i-1}\, |v| +|\partial_i u|^{p_i-1}\, |u|)\, dx\, d\tau  \\
    & \le \frac{C}{R}\sum_{i=1}^N \|\partial_i v\|^{p_i}_{L^{p_i}(S_T)}+\|v\|^{p_i}_{L^{p_i}(S_T)}+\|\partial_i u\|^{p_i}_{L^{p_i}(S_T)}+ \|u\|^{ p_i}_{L^{p_i}(S_T)}
\end{split}
\]
by Young's inequality. By the initial argument and the assumptions, the sum on the right is finite, while by   the monotonicity of the operator both terms on the left are non-negative. Hence  for any $t<T$ the left hand side vanishes for $R\to +\infty$, giving the claim.
\end{proof}
 \noindent As a corollary, we have the following comparison principle for solutions to the Fokker-Planck equation.
\begin{corollary} \label{comparison4life}
Let $v, w$ be $L^{\bf{p}}$-solutions of the Cauchy problem for the Fokker-Planck equation \eqref{FokkerPlanck} satisfying $w_{0}\ge v_{0}$ and $w_{0}, v_{0} \in L^{2}(\R^{N})$. Then $w\ge v$ in $S_T$.
\end{corollary}

\begin{proof}
It suffices to recall that the law \eqref{continuous-transformation} establishes an isomorphism between  $L^{{\bf p}}(S_{1,e^{T}})$-solutions of the prototype equation \eqref{mod}  and $L^{{\bf p}}(S_{T})$ solutions to the Fokker-Planck equation \eqref{FokkerPlanck} and the initial data coincide.
\end{proof}

\begin{remark} 
It is worth noting that, while an elliptic comparison principle holds true as well for the stationary counterpart of \eqref{mod}, this is no longer the case for the stationary counterpart of the Fokker-Planck equation \eqref{FokkerPlanck}. This can be seen considering (already in the isotropic case), the  Barenblatt solutions of the $p$-Laplacian equation, which solve the stationary Fokker-Planck equation and contradict the elliptic comparison principle for \eqref{FokkerPlanck}.
\end{remark}

\subsection{$L^{{\bf p}}$ solutions} \label{esistenza}
We next consider the Cauchy problem for \eqref{mod}, with bounded and a  compactly supported initial datum, attained in  $L^2$. This problem can be read as
\begin{equation} \label{EQT}
    \begin{cases}
    \partial_{t}u=\sum_{i=1}^N \partial_i \big( (\partial_i u)^{p_i-1}  \big) &\text{in $S_{T}$},\\
    u_{0}=g \in L^{2}(\R^N)&\text{supp}\, g \subset \bar{B}_{R_0}, \quad \quad g \in L^{\infty}(B_{R_0}).
    \end{cases} 
\end{equation}  \noindent
We show in this section that this problem has a unique $L^{{\bf p}}$-solution, by a standard approximation technique relying on the monotonicity of the operator.
\begin{proposition}
\label{puni}
Problem \eqref{EQT} has a unique $L^{{\bf p}}$-solution which takes $g$ as initial datum in $L^2$.
\end{proposition}
\begin{proof}
We let, for $n\ge {\rm diam} ({\rm supp}\, w_{0})$, $B_n=\{|x|< n\}$ and  consider the boundary value problems
\begin{equation} \label{vn}
    \begin{cases}
    v_n \in C(0,T; L^2(B_n)) \cap L^{{\bf p}}(0,T ; W^{1,{\bf p}}_0(B_n))\\
    \partial_{t}v_n - \sum_{i=1}^N \partial_i ((\partial_i v_n)^{p_i-1})=0,\quad \text{in} \quad B_n \times (0,T),\\
    v_n(\cdot,0)=g.
    \end{cases}
\end{equation} \noindent 
We regard them as defined in the whole $S_{T}$ by extending them to be zero on $|x|>n$. The problems \eqref{vn} can be uniquely solved by a monotonicity method (see for instance \cite[Example 1.7.1]{Lions}), and give solutions $v_n$ satisfying
\begin{equation} \label{L2}
\sup_{t \in [0,T]} \int_{\R^N} |v_n(x, t)|^2 \, dx + 2 \sum_{i=1}^N  \int_{S_T} | \partial_i v_n |^{p_i} \, dx\, dt = \|g\|^2_{2}, \quad \forall n \in \mathbb{N},
\end{equation} \noindent and thus $v_n \in L^{\infty}(0,T; L^2(\R^N))$ uniformly. Notice that by the local comparison principle in $\Omega_{T}$, $\|v_{n}\|_{\infty}\le \|g\|_{\infty}$ hence by dominated convergence $v_{n}(\cdot, t)\to g$ in $L^{2}(\R^{N})$ implies $v_{n}(\cdot, t)\to g$ in $L^{p_{i}}(\R^{N})$ as $t\to 0$, for $i=1, \dots, N$. In the weak formulation of \eqref{vn} we take (modulo a Steklov averaging process) the test function $(v_n)^{p_j-1}$, $j=1,  \dots, N$, obtaining $\forall t \in (0,T)$
\[
\int_{\R^N}\frac{|v_n|^{p_j}}{p_j}(x, t) \,dx+ (p_j-1)\sum_{i=1}^N  \int_{S_{t}} |\partial_i v_n|^{p_i} \,  |v_n|^{p_j-2} \, dx d\tau=\int_{\R^{N}}\frac{|g|^{p_j}}{p_j} \,dx.
\]
implying
\begin{equation}
\label{lpi}
 v_n \in \cap_{i=1}^N L^{\infty}(0,T; L^{p_i}(\R^N)), \quad \text{with a uniform bound}.
\end{equation}
This estimate, together with \eqref{L2}, provides an uniform bound for $v_n$ in 
\[
L^{{\bf p}}(0,T; W^{1,\bf{p}}(\R^N))\cap L^{\infty}(0, T; L^{2}(\R^{N})).
\]
This bound implies that a (not relabelled) subsequence $v_n$ converges weakly* to a function $v$ in these spaces. Moreover, the weak formulation of the equation implies that
 \[
 \partial_{t}v_{n}=\sum_{i=1}^{N}\partial_{i}((\partial_{i} v_n)^{p_{i}-1}),
 \]
 and for any $m\in N$ the right hand side is uniformly bounded in 
 \[
 \big( L^{\bf{p}}\big(0, T; W^{1, \bf{p}}_0(B_{m})\big) \big)' =:L^{{\bf p}'}(0, T; W^{-1, {\bf p}'}(B_{m}))
 \]
 by H\"older inequality. By Aubin-Lions theorem \cite[Chap. III Proposition 1.3]{Showalter}, applied to the triple
\[
W^{1, \bf{p}}_{0}(B_{m})\hookrightarrow L^{2}(B_{m})\to W^{-1, \bf{p}'}(B_{m}),
\]
we can select for each $m$ a subsequence $v_{n}$ that converges to a function $v$ in $L^{2}(0, T; L^{2}(B_{m}))$. A diagonal argument provides a subsequence (still not relabeled) converging in $L^{2}(0, T; L^{2}_{loc}(\R^N))$ to the weak* limit $v$ and such that 
\begin{enumerate}
\item 
$\displaystyle{\int_{\R^{N}} (v_{n})_{t}\, \varphi_{t}\, dx\to \int_{\R^{N}} v_{t}\, \varphi_{t}\, dx \qquad}$ for a.e. $t$ and all $\varphi\in C^{\infty}_{{\rm loc}}(0, T; C^{\infty}_{c}(\R^{N}))$,\vskip8pt
\item $ \partial_i(\partial_i v_n)^{p_i-1}  \rightharpoonup \eta_i\,$, weakly in $L^{ p_i}(S_T)$ for some $\eta_{i}$,\quad  $\forall i=1, \dots, N$.
\end{enumerate}
We can thus pass to the limit in the weak formulation of the equation, identifying $\eta_{i}=(\partial_{i} v)^{p_{i}-1}$ through Minty's trick. Semicontinuity and \eqref{lpi} imply that $v$ is an $L^{{\bf p}}$ solution.
\newline 
In order to prove uniqueness, let $u_1,u_2$ be two  $L^{{\bf p}}$-solutions of \eqref{EQT}. By the first step of the proof of Proposition \ref{com}, both belong to $L^{{\bf p}}(0, T; W^{1, {\bf p}}(\R^{N}))$, thus $v:=u_1-u_2$ satisfies
\[
    \begin{cases}
    v \in C(0,T ; L^2(\R^N)) \cap L^{{\bf{p}}}(0,T; W^{1,{\bf{p}}}(\R^N)),\\
    \partial_{t}v=\sum_{i=1}^N \partial_{i}\big((\partial_i u_1)^{p_i-1}   -(\partial_i u_2)^{p_i-1}\big), \quad \text{in} \quad S_{T},\\
    v_{0}=0.
    \end{cases}
\]
Test the latter with  $v\, \zeta$, where $\zeta\in C^{\infty}_{c}(B_{2R})$, $\zeta\ge 0$, $\zeta=1$ in $B_R$ and $|D \zeta | \le C/R$. For all $0<t \le T$ we have
\[
    \begin{split}
 \frac{1}{2} \int_{B_R} |v_{t}|^2 \,dx& + \int_0^t \int_{B_{2R}} \sum_{i=1}^N \big((\partial_i u_1)^{p_i-1}   -(\partial_i u_2)^{p_i-1}\big)\, ( \partial_i u_1-\partial_i u_2 \big)\,  \zeta \, dxd\tau\\
&=- \int_0^t \int_{B_{2R}}\sum_{i=1}^N v\, \big((\partial_i u_1)^{p_i-1}   -(\partial_i u_2)^{p_i-1}\big)\,  \partial_i  \zeta \, dxd\tau.
\end{split}
\]
Using  the monotonicity of the principal part on the left-hand side and H\"older's inequality on the right, for every $t\in (0, T)$
\[
\int_{B_R} |v_{t}|^2 \, dx \le\frac{C}{R} \sum_{i=1}^N \|v\|_{L^{p_i}(S_{T})}\, (\|\partial_i u_1\|_{L^{p_i}(S_{T})}+\|\partial_i u_2\|_{L^{p_i}(S_{T})} )  \rightarrow 0, \quad \text{as} \quad R \rightarrow \infty.
\]
\end{proof}

\section{ Barenblatt fundamental solutions}\label{bare}

In this section we will build a self similar solution to \eqref{mod}, i.e., by the discussion in section \ref{scaling}, a stationary solution to the Fokker-Planck equation \eqref{FokkerPlanck}. We will then study the positivity properties of such fundamental solution, which, together with the comparison principle, will be the main tool to expand the positivity set of non-negative solutions of \eqref{mod}.\newline 

By the results of section \ref{esistenza} we can define,  at least for bounded compactly supported initial data $g$, the operator 
\[
\S_{t} g :=u_{t},\qquad t\ge 1,
\]
where $u$ is the unique $L^{{\bf p}}$ solution of 
\begin{equation}
\label{EQ}
\begin{cases}
    \partial_{t}u=\sum_{i=1}^N \partial_i \big( (\partial_i u)^{p_i-1}  \big) &\text{in $S_{1, \infty}$},\\
    u_{1}=g.&
    \end{cases} 
\end{equation}
In terms of the Fokker-Planck equation, this also defines through \eqref{continuous-transformation} the operator
\begin{equation}
\label{stilde}
\tilde{\S}_s g:= (\Phi u)_s\qquad s\ge 0,
\end{equation}
giving the solution at the time $s \in \R_+$, of the problem 
 \begin{equation} 
 \label{FKP}
    \begin{cases}
    \partial_{s}w= \sum_{i=1}^{N}\partial_i [ (\partial_i w )^{p_i-1} -\alpha_i y_i w ] &\text{in $S_{\infty}$},\\
    w_{0}=g. 
    \end{cases} 
\end{equation} \noindent 
The relation \eqref{stilde} implies that
\begin{equation}
\label{hatT}
\tilde{\S}_{s}g=\T_{e^{s/\sigma}}\S_{e^{s}}g,
\end{equation}
where $\T$ is given in \eqref{trho}, allowing  to prove properties for $\tilde{\S}_s$ by proving them  for $\S_t$.

\subsection{Construction of a Barenblatt solution}

In order to state some basic properties of the operator $\tilde{\S}_{s}$ we will need the following space:
\begin{equation}
\label{X}
X_{R, M}=\{g\in L^{\infty}(\R^{N}):  0\le g\le M, \ {\rm supp}\, g \subseteq K_{R}\}, \qquad X=\bigcup_{R, M>0}X_{R,M}.
\end{equation}

\begin{lemma}\label{lemma31}
If  \eqref{param} holds true, the operator  $\tilde{\S}_{s}$, $s\ge 0$  defined in \eqref{stilde} has the following properties.
\begin{enumerate}
\item If  $g\in  L^{2}(\R^{N})$ and ${\rm supp}\, g\subseteq K_{R_{0}}$ then for some $c=c(N, {\bf p})$ it holds
\begin{equation}
\label{supp}
{\rm supp}\, \tilde{\S}_{s}g\subseteq \prod_{i=1}^{N} [-R_{i}(s), R_{i}(s)],\qquad R_{i}(s)=2\,e^{-s\alpha_{i}}R_{0}+c\,\|g\|_{1}^{ \bar{p}(p_i-2)/(p_i \sigma)}.
\end{equation}
\item If $g\in X$, then
$\|\tilde{\S}_{s}g\|_{1}=\|g\|_{1}$ and $0\le \tilde{\S}_{s}g\le \|g\|_{\infty}$. In particular $\tilde{\S}_{s}:X\to X$ for all $ s\ge 0$.
\item
For any $R, M>0$ and $s\ge 0$, $\tilde{\S}_{s}:X_{R, M}\to X$ is continuous when $X_{R, M}$ and $X$ are equipped with the weak-$L^{2}$ topology.
\end{enumerate}

\end{lemma}

\begin{proof}
Consider the corresponding problem \eqref{EQ} and the therein defined operator $\S_{t}$. By \cite[Theorem 1.1]{Mosconi} (notice that the branch obtained there is an $L^{{\bf p}}$ solution and therefore coincides with $\S_{t}g$ by uniqueness) we know that if ${\rm supp}\, g\subseteq K_{R_{0}}$, then 
\begin{equation}
\label{csup}
\text{supp}\, \S_{t} g \subseteq \prod_{i=1}^{N}[-R_{i}(t), R_{i}(t)],  \qquad R_{i}(t) =2\, R_{0}+c\, (t-1)^{\alpha_{i}}\, \|g\|_{1}^{ \bar{p}(p_i-2)/(p_i \sigma)}.
\end{equation}
Letting $t=e^{s}$ and using  \eqref{hatT} we get the first assertion, since
\[
\text{supp}\, \tilde{\S}_{s} g\subseteq \prod_{i=1}^{N}[-\tilde{R}_{i}(s), \tilde{R}_{i}(s)],\qquad \tilde{R}_{i}(s)=e^{-s\alpha_{i}}R_{i}(e^{s})\le 2\, e^{-s\alpha_{i}}R_{0}+c\, \|g\|_{1}^{ \bar{p}(p_i-2)/(p_i \sigma)}.
\]

The second statement follows from its counterpart on the corresponding solution $u$ of \eqref{EQ}: to prove conservation of mass we take advantage of the compactness of the supports of $u$ dictated by \eqref{csup} and test \eqref{EQ} with $\varphi\in C^{\infty}_{c}(\R^{N})$ such that $\varphi\equiv 1$ on $\cup_{t<T}{\rm supp}\, u_{t}$, $T>0$ arbitrary. The point-wise bounds follow from the local comparison principle for \eqref{EQ}, again taking advantage of the compactness of the support and comparing $u$ with the solutions $v\equiv 0$ and $v\equiv \|g\|_{\infty}$, respectively.

It remains to prove the continuity of $\tilde{\S}_s : X_{R, M} \rightarrow X$ within the weak $L^{2}$ topologies  from departure to arrival, which by \eqref{hatT} is equivalent to prove the same statement for $\S_{t}$.
Fix $T>t\ge 1$ and let 
\[
\bar R=\max\big\{2\, R+C\, (T-1)^{\alpha_{i}}\, (|K_{R}|M)^{ \bar{p}(p_i-2)/(p_i \sigma)}:i=1, \dots, N\big\}.
\]
Assume $g_{n}\to g$ weakly in $L^2$ with $g_{n}\in X_{R, M}$ and let $u_{n}$ be the $L^{{\bf p}}$ solution of \eqref{EQ} with initial data $g_{n}$. Notice that thanks to \eqref{csup}, it holds ${\rm supp}\, (u_{n})_{\tau}\subseteq K_{\bar R}$ for every $\tau\in [0, T]$, $n\ge 1$. The boundedness of  $\|g_{n}\|_{2}$ and  standard  energy estimates then give a uniform bound for $u_{n}$ in $ L^{{\bf p}}(1,T; W^{1,{\bf p}}_0(K_{\bar R})) \cap L^\infty(1,T;L^2(\R^N))$ and for $\partial_{\tau}u_{n}$ in $L^{{\bf p}'}(0, T; W^{-1, {\bf p}'}(K_{\bar R}))$. Applying Aubin-Lions theorem as in  the proof of Proposition \ref{puni}, we can extract a subsequence converging weakly$^{*}$ to some $u$ in those spaces and such that  
\[
\text{$u_n(\cdot, \tau)\to u(\cdot, \tau)\qquad $ in $L^{2}(K_{\bar R})$, for a.\,e. $\tau \in [1, T]$.}
\]
We can pass to the limit in the weak form of the equation to get 
\[
 \int_{\R^{N}} u_{\tau}\, \varphi_{\tau}\, dx- \int_{\R^{N}} g\, \varphi_{1}\, dx-\int_{S_{1,\tau}}u\, \partial_{\tau}\varphi\, dx\, dt+\int_{S_{1, \tau}}\sum_{i=1}^{N}\eta_{i}\, \partial_{i}\varphi\, dx\, dt=0
 \]
for almost every $1<\tau<T$, so that it only remains to show that $\eta_{i}=(\partial_{i}u)^{p_{i}-1}$. We cannot directly employ Minty's trick, since we are missing the strong convergence of the initial data. However,  for any $\tau$ such that $(u_n)_{\tau}\to u_{\tau}$ in $L^{2}(K_{\bar R})$, we look at $\{u_{n}\}$ as a sequence of solution to \eqref{mod} on $[\tau, T]$ with strongly convergent initial data and now Minty's trick allows to deduce $\eta_{i}=(\partial_{i} u)^{p-1}$ on $S_{\tau, T}$. Since $\tau$ can be chosen arbitrarily close to $1$ we obtain that $u$ is a $L^{{\bf p}}$ solution to \eqref{EQ} with initial datum $g$ and from uniqueness we infer that $u_{t}=\S_{t}g$ for any $t\ge 1$. A standard sub-subsequence argument concludes the proof of the third statement. 
 \end{proof}

\begin{theorem}
Under assumption \eqref{param}, there exists a nontrivial stationary solution $w\in X_{1,1}$ to \eqref{FKP}, and therefore a Barenblatt Fundamental solution.
\end{theorem}

 \begin{proof}
For $R_{0}, M_{0}> 0$ consider the convex set 
 \[
 C_{\eps}:= \big\{ g \in L^2(\R^N): \text{supp}\, g \subset B_{1},\, 0 \le g \le 1, \|g\|_{L^1(\R^N)}=\eps \big\}\subseteq X_{1,1}.
 \]
If $c$ is given in \eqref{supp}, for $\bar s$ sufficiently large and $\bar \eps$ sufficiently small it holds 
\[
2\, e^{-\bar s\alpha_{i}}+c\, {\bar\eps}^{ \bar{p}(p_i-2)/(p_i \sigma)}\le 1\qquad \forall i=1, \dots, N,
\]
 so that  \eqref{supp} implies that ${\rm supp}\, \tilde{\S}_{\bar s}g\subseteq B_{1}$ for all $g \in C_{\bar{\eps}}$. Using also point (2) of the previous lemma we have that $\tilde{\S}_{\bar s} C_{\bar \eps} \subseteq C_{\bar \eps}$. Moreover,  $C_{\bar \eps}$ with the weak $L^{2}$ topology  is compact, and by point (3) of the previous lemma, $\tilde{\S}_{\bar s}:C_{\bar \eps}\to C_{\bar \eps}$ is continuous, so that Schauder's theorem ensures the existence of a fixed point $\bar g\in C_{\bar \eps}$ for $\tilde{\S}_{\bar s}$. Therefore the function $\bar w_{s}=\tilde{S}_{s}\bar g$ is a times-periodic, bounded and compactly supported solution of  \eqref{FKP}, which can therefore be extended to $\R^{N+1}$ as an aeternal solution. Consider the  bounded, compactly supported function
\[
g(y)= \sup_{s\in \R} \bar w (s,y), \quad g \in X_{1,1}\, ,
\]
 for which $\|g\|_{1}\ge \bar{\eps}$. Then  $\tilde{\S}_{0}g=g\ge \bar w _{\tau}$ for every $\tau\in \R$, so that by the comparison principle \ref{comparison4life} it holds $\tilde{\S}_{s} g\ge \bar w_{\tau+s}$ for any $s\ge 0$. Taking the supremum in $\tau\in \R$ gives $\tilde{\S}_{s}g \ge g$, but since $\|\tilde{\S}_{s}g\|_{1}=\|g\|_{1}$, this implies $\tilde{\S}_{s}g= g$ for every $s\ge 0$, i.e., $g$ is a stationary solution of \eqref{FKP}.
\end{proof}

\subsection{Properties of the Barenblatt solutions}

Our next aim  is to prove that Barenblatt solution are positive in a quantitative way, i.e., their positivity set spreads in time in a way controlled by scaling.  This amounts in proving that stationary non-negative solutions of the Fokker Planck equation are bounded from below near the origin, which is the content of the next theorem. 

\begin{theorem}
Suppose \eqref{param} holds,  let $w\in X_{1,1}$  (see \eqref{X}) be a nontrivial stationary solution of the Fokker-Planck equation \eqref{FokkerPlanck} and $\B$ the corresponding  Barenblatt  solution of \eqref{mod}. Then there exists $ \bar\eta>0$, depending on $w$ and the data, such that
\[
\B(x, t)\ge \bar \eta\, t^{-\alpha} \qquad \text{if \ \ $|x_i| < \bar \eta\, t^{\alpha_i}$ for $i=1, \dots, N$}.
\]
\end{theorem}
\begin{proof}
Suppose that $\B$ is given by 
\begin{equation} \label{law}
\B(x,t)= t^{-\alpha}\,  w(x_i\,  t^{-\alpha_i}), \qquad t\ge 1.
\end{equation}
By  \cite[Corollary 4.3]{Mosconi}  we can fix a lower-semicontinuous representative of $\B$ and thus of $w$. Since $w>0$ somewhere, we can pick a point  $x^{(0)}$ and numbers $\delta_{0},\eta_{0}>0$ such that  
\begin{equation} \label{positivity1}
    \inf_{K_{\delta_{0}}(x^{(0)})}w(y)>\eta_{0}.
\end{equation} 
By \eqref{law}, the latter implies for any $t\ge 1$
\[
\B(x, t) \ge \eta_{0}\,  t^{-\alpha}, \quad \text{when} \quad \big\{ |x_i-x^{(0)}_{i}\, t^{\alpha_i}| <\frac{\delta_0}{2}\,  t^{\alpha_i} \big\}.
\]
Consider now 
\[
    \B_{\lambda}(x, t)= \lambda\,  t^{-\alpha}\,  w \big( \lambda^{(2-p_i)/p_i}\,  t^{-\alpha_{i}}\, (x_i^{(0)}-x_i)  \big),\\
\]
which solves \eqref{mod} by translation invariance and Proposition \ref{transformation-group}. Notice that, since $w\in X_{1,1}$
\[
\|\B_{\lambda}(\cdot, t)\|_{\infty}=\lambda\,  t^{-\alpha}\qquad {\rm supp}\, \B_{\lambda}(\cdot, t) \subseteq \big\{2\, |x_{i}^{(0)}-x_{i}|\le t^{\alpha_{i}}\, \lambda^{(p_i-2)/p_i}\big\}.
\]
We seek for  $\lambda>0$ such that the comparison principle can be applied between $\B_{\lambda}$ and $\B$ with starting time $t=1$. We need
\[
    \begin{cases}
   \| \B_{\lambda}(\cdot, 1)\|_{\infty}\le \eta_{0},\\
    \text{supp}\,  \B_{\lambda}(\cdot, 1)\subseteq K_{\delta_{0}}(x^{(0)}),
    \end{cases} \iff\quad  \begin{cases}
    \lambda \le \eta_{0},\\
   \lambda^{(p_i-2)/p_i} \le \delta_{0}/2 ,
    \end{cases}
\]
which, being $p_{i}>2$ for all $i$, can be solved for some $\lambda=\lambda_{1}\in (0, 1)$. Consequently, by comparison and \eqref{positivity1}, there holds 
\[
 \B(x,t) \ge \B_{\lambda_1}(x,t) > \lambda_1\, t^{-\alpha}\,  \eta_{0} , \qquad \text{for}\qquad \big| x^{(0)}_i - \lambda_1^{(2-p_i)/p_i}\,  t^{-\alpha_{i}}\, (x_i^{(0)}-x_i) \big| <\frac{\delta_{0}}{2}.
\] 
We let $t_1^{\alpha_1}= \lambda_1^{(2-p_1)/p_1}\ge 1$ and, consequently,
\[
 \eta_{1}=\lambda_1\, t_{1}^{-\alpha}\,  \eta_{0}, \quad x^{(1)}_{i}:=x^{(0)}_i \big(1-t_1^{\alpha_i}\,  \lambda_1^{(p_i-2)/p_i}\big),\quad \delta_1:= \frac{\delta_{0}}{2}  \min \big\{ t_1^{\alpha_i} \, \lambda_1^{(p_i-2)/p_i}:i=1, \dots, N\big\} 
\]
(notice that, by the choice of $t_{1}$, it holds $x^{(1)}_{1}=0$), to get 
\[
\inf_{K_{\delta_{1}}(x^{(1)})}\B(\cdot, t_1) \ge \eta_1
\]
Proceeding by induction, we will find sequences $t_{n}, \eta_{n}, \delta_{n}, x^{(n)}$ with the properties
\[
\inf_{K_{\delta_{n}}(x^{(n)})}\B(\cdot, t_n) \ge \eta_n,\qquad x^{(n)}_{i}=0\quad \text{for $i=1, \dots, n$}
\]
so that after $N$ steps $x^(N)=0$ and we find
\[
\inf_{K_{\delta_{N}}}\B(\cdot, t_N)\ge \eta_N.
\]
By \eqref{law}, this implies $w(x)\ge \eta_{N}\, t_{N}^{\alpha}$ when $|x_{i}|<t_N^{\alpha_{i}}\, \delta_{N}/2 $ for $i=1,\dots, N$. We set $\bar \eta=\min\{\eta_{N}, \delta_{N}/2\}$ and scale back to $\B$ through \eqref{law} again, to get the claim.
\end{proof}

We will from now suppose that $w$ is a fixed stationary solution in $X_{1,1}$ of \eqref{FKP}. For future purposes we summarise some properties derived from a scaling argument for a large family of corresponding Barenblatt solutions.

\begin{corollary}
\label{corB}
Let $\B(x,t)= t^{-\alpha}w(x_it^{-\alpha_i})$ be a fixed Barenblatt Fundamental solution to \eqref{mod} with $w\in X_{1,1}$. There exists $\bar \eta>0$ such that the family of Barenblatt solutions 
\[
\B_{\lambda}(x, t)= \T_{1,\lambda^{-\sigma/\bar p}} \B\, (x, t)=\lambda\,  t^{-\alpha}\, w(\lambda^{(2-p_{i})/p_{i}}\, x_{i} \, t^{-\alpha_{i}}),\qquad \lambda>0,
\]
has the following properties
\begin{enumerate}
\item $\displaystyle{\|B_{\lambda}(\cdot, t)\|_{\infty}=\lambda \,  t^{-\alpha}}$;
\item
$\displaystyle{{\rm supp}\, B_{\lambda}(\cdot, t)\subseteq \prod_{i=1}^N \big{\{} |x_i|\le    \lambda^{(p_i-2)/p_i}\, t^{\alpha_i} \big{\}}}$;
\item
$\displaystyle{\{B_{\lambda}(\cdot, t)\ge \lambda \, t^{-\alpha}\}\supseteq \prod_{i=1}^N \big{\{} |x_i|\le \bar\eta\, \lambda^{(p_{i}-2)/p_{i}}\, t^{\alpha_i} \big{\}}}$.
\end{enumerate}
\end{corollary}

\section{Proof of Theorem \ref{Harnack-Inequality}}\label{final}

 \noindent We first consider a generalisation  of what is called in literature the Krylov-Safonov argument. To this end, we make the following observations: for $\rho \in [0,1]$ the translates of the cylinders $\Q^{-}_{\rho}(\rho^{-N})$ arise naturally from the {\em quasi-metric}\footnote{This terminology is borrowed from Grafakos, but it appears there's no general consensus on the term ``quasi'': sometimes {\em pseudo-metric} is used instead.}
 \begin{equation}
 \label{metric}
 {\rm d}((x, t), (y, s))=\max\left\{ |2^{-1}(x_{i}-y_{i})|^{p_i/(\bar{p} + N(\bar{p}-p_i))}, |t-s|^{1/(\bar{p} + N(\bar{p}-2))}\right\}.
 \end{equation}
Indeed, all the exponents appearing in the previous definition are positive thanks to condition \eqref{param} on the spareness of $p_i$'s, therefore the {\em quasi-triangle inequality}
 \[
 {\rm d}(z_{1}, z_{3})\le \gamma\, ({\rm d}(z_{1}, z_{2})+{\rm d}(z_{2}, z_{3})),\qquad \forall z_{1}, z_{2}, z_{3}\in \R^{N+1},
 \]
 holds true for a constant $\gamma=\gamma(N, {\bf p})\ge 1$ which is the {\em quasi-metric constant}. Finally, notice that the  cylinder $\bar z+\Q^{-}_{\rho}(\rho^{-N})$ is the bottom half part of the ball $B_{\rho}(\bar z)$ with respect to this distance. 
\begin{lemma}\label{Krylov-Safonov}
Let $(X, {\rm d})$ be a quasi-metric space with quasi-metric constant $\gamma$  and $x_{0}\in X$. For any $\beta>0$ there exists a constant $\omega=\omega(\gamma, \beta)>1$ such that for any  bounded $u:B_{1}(x_{0})\to \R$ with $u(x_{0})\ge 1$ there exist $ x\in B_{1}(x_{0})$ and $r>0$ such that 
\begin{equation}\label{stime}
B_{r}(x)\subseteq B_{1}(x_{0}),\qquad     r^{\beta}  \sup_{B_{r}(x)} u\le \omega ,\qquad r^{\beta} u(x)\ge 1/\omega.
\end{equation}
\end{lemma}

\begin{proof} 
Extend $u$ as 0 outside $B_{1}(x_{0})$ and suppose that the claim is false. For $\omega$ a parameter to be determined depending only on $\beta$ and $\gamma$, we will construct  a sequence of points contradicting the boundedness of $u$. Set $r_{0}=1/(2\gamma)$ and choose $\omega>(2\gamma)^{\beta}$. Since $r_{0}^{\beta}u(x_{0})\ge 1/\omega$, it must hold
\[
r_{0}^{\beta}  \sup_{B_{r_{0}}(x_{0})} u> \omega.
\]
Choose $x_{1}\in B_{r_{0}}(x_{0})$ such that $r_{0}^{\beta}\, u(x_{1})\ge \omega$ and set $r_{1}=r_{0}\, \omega^{-2/\beta}$, so that
\[
r_{1}^{\beta}\,  u(x_{1})\ge 1/\omega.
\]
If $B_{r_{1}}(x_{1})\subseteq B_{1}(x_{0})$, we can similarly construct $x_{2}\in B_{r_{1}}(x_{1})$ such that
\[
r_{2}^{\beta}\,  u(x_{2})\ge 1/\omega,\qquad r_{2}=r_{1}\, \omega^{-2/\beta}.
\]
Proceed by induction to get a sequence of points and radii such that, if $B_{r_{n}}(x_{n})\subseteq B_{1}(x_{0})$,
\[
 r_{n}^{\beta}\, u(x_{n})\ge 1/\omega,\qquad r_{n}=r_{n-1}\, \omega^{-2/\beta}.
\]
As $\omega>1$, the first condition contradicts the boundedness of $u$ if all the balls $B_{r_{n}}(x_{n})$ are contained in $B_{1}(x_{0})$. This can be achieved if for any $n\ge 0$
\[
{\rm d}(x_{0}, x_{n})\le\gamma \sum_{i=0}^{n-1}\gamma^{i}\, {\rm d}(x_{i}, x_{i+1})\le \gamma r_{0} \sum_{i=0}^{+\infty}\gamma^{i}\,  \omega^{-2i/\beta }<1,
\]
which holds for $\gamma \, \omega^{-2/\beta}<1/2$.
\end{proof} 
\begin{lemma} \label{clusteringpotente}
Let $u\ge 0$ solve \eqref{mod} in $Q_{1}^{-}$, and suppose that for some $\bar\nu\in (0, 1)$ $a>0$ it holds
\begin{equation}
\label{condizionepotente}
    |[u>a] \cap Q_{1}^{-}| > (1-\bar\nu)\,  | Q_{1}^{-}|. 
\end{equation}
Then for every choice of $\lambda,\nu \in (0,1)$ there exist  $ \bar y\in K_{1}$, $\bar t\in (-1, -\bar{\nu}/4]$ and $\e \in (0,1)$ determined only by means of $N,{\bf p},\nu, \bar\nu, a$ and $\lambda$,  such that $\bar y+ K_{\e } \subset K_{1}$ and
\begin{equation} \label{clusterino}
    |[u_{\bar{t}} > \lambda\,  a] \cap  (\bar y+K_{\e})| > (1-\nu)\,  |K_{\e}|.
\end{equation}
\end{lemma}
\begin{proof}
Choose $r=r(\bar\nu)>1/2$ sufficiently near $1$ so that $|[u>a]\cap Q_{r}|>|Q_{r}|\, (1-\bar\nu)/2$. We write down the energy estimates \eqref{energy-estimates} for $(u-a)_-$ with $\eta$ of the form prescribed therein, $\eta\ge 0$, $\eta=1$  on $Q_{r}^{-}$, $\eta=0$ outside $Q_{1}^{-}$ and $|\partial_t \eta|+|\partial_{i} \eta_i | \le C(\bar\nu)$, to get, thanks to $\sup_{Q_{1}^{-}}(u-a)_{-}\le a$,
\[
    \sum_{i=1}^N \int_{Q_{r}^- }  |\partial_i (u-a)_{-}|^{p_i} dx\, dt \le C( \bar\nu, a). 
\]
By  the same argument  of \cite[Lemma 9.1]{DGV-acta}, there exists a time level $\bar{t} \in (-1,-\bar{\nu}/4]$ such that
\[
    \sum_{i=1}^N \int_{K_r} \bigg|\partial_i  \bigg(\frac{u_{\bar t}}{a}-1\bigg)_{-}\bigg|^{p_i} dx \le C(\bar \nu, a)\, r^{N-1} ,  \qquad \bigg| \bigg[\frac{u_{\bar t}}{a}>1 \bigg] \cap K_r \bigg| \ge  (1-\bar\nu)\, |K_{r}|/4.
\]
By H\"older's inequality, $u_{\bar t}/a$  fullfills the assumptions of Lemma \ref{localclustering-revisited} in $K_{r}$, giving the claim.
\end{proof}
\noindent It is worth underlining that the parameter $\epsilon$ in the previous statement can be made arbitrarily small by eventually changing the point. We further observe that it is possible to carry the information of Lemma \ref{clusteringpotente} into an equivalent formulation in the anisotropic cubes $\Q_{\rho}(\rho^{-N})$ by using \eqref{trho}.

\noindent In the next Lemma, we suppose that an essentially upper semicontinuous representative for the solution has been chosen,  through \cite[Corollary 4.3]{Mosconi}.

\begin{lemma} \label{proposizione-potente}
Let $u\ge 0$ be a bounded solution of \eqref{mod} in $Q_1^{-}$. There exist $C, \varepsilon>0$ depending on $N$ and ${\bf p}$ such that if $u(0,0)\ge C$, 
\begin{equation}\label{infestimate}
    \inf_{\bar x+\K_{\rho}(\varepsilon \rho^{-N})}u_{\bar t} \ge \varepsilon\, \rho^{-N} \qquad \text{for some $(\bar{x}, \bar{t})\in Q_{1}^{-}$ and $\rho>0$ with $ \bar{x}+\K_{\rho}(\varepsilon\,  \rho^{-N})\subseteq K_{1}$}.
\end{equation}
\end{lemma}

\begin{proof}
Let $C=1/\omega$, where $\omega=\omega(N, {\bf p})$ is given in  Lemma \ref{Krylov-Safonov}  with $\beta=N$, (using the quasi-metric in \eqref{metric}). We apply the lemma  to $u/C$ and extend $u$ as $0$ in the upper half-space. Then, \eqref{stime}  implies the existence of a point $ z_{1}\in Q_{1}^{-}$ and $r\in (0, 1)$ such that
\[
z_{1}+\Q_{r}^{-}(r^{-N})\subseteq \Q_{1}^{-},\qquad r^{N}\sup_{ z_{1}+\Q_{r}^{-}(r^{-N})}u\le 1 ,\qquad r^{N}\, u( z_{1})\ge C^{2}.
\]
The solution $v=\T_{r, r^{-N}}u(\cdot+z_{1})$ in $Q_{1}^{-}$, (with $\T$ given in \eqref{transformation}) obeys
\begin{equation}
\label{temp}
\sup_{Q_{1}^{-}}v\le 1,\qquad  v(0)\ge C^{2}.
\end{equation}
 We prove that \eqref{condizionepotente} holds for $\bar \nu=1-\mu_{a}$ given in Lemma \ref{DG} when $a=C^{2}/3$ (thus $\bar \nu$ depends only on $N$ and ${\bf p}$).   Indeed,  if,  by contradiction, we have
\[
 |[v \ge a]\cap Q_{1}^-| \le \mu_{a}\, |Q_{1}^-|,
\]
then since $0\le v\le 1$ in $Q_{1}^-$, Lemma \ref{DG} gives 
\[
v(0)\le \sup_{Q_{1/2}^{-}}v \le \frac{3}{2}\, a= \frac{C^{2}}{2},
\]
contradicting the last condition in \eqref{temp}. Therefore the thesis of Lemma \ref{clusteringpotente} holds true for any $\nu, \lambda$ to be chosen and for the corresponding point $z_{2}=(\bar x, \bar t)\in Q_{1}^{-}$, the following measure estimate holds  at the time $\bar{t}$
\[
  |[v_{\bar{t}} \le \lambda\, a ] \cap (\bar x+ K_{\e})| \le \nu\,  |K_{\e}|,\qquad  \bar t \in (-1,-{\bar{\nu}}/{4}].
\]
Recall that this measure estimate is valid for any $\nu,\lambda>0$ to be chosen, which in turn determine an arbitrarily small $\e$, so we can also suppose
\[
 \nu< \bar{\nu},\qquad \e^{-2}\bar{\nu}>1, \qquad 
\e^{-1}a> 2,
\]
where $a=C^{2}/3$.
We choose $\lambda=1/2$ and scale again considering $w=\T_{\e/2, \e/2}v(\cdot+z_{2})$. Since $v$ solves \eqref{mod} in $Q_1^-$ and by \eqref{tr} it holds $K_{\e}=\K_{\e}(\e)=T_{\e/2, \e/2}(K_{2})$,  $w$ solves \eqref{mod} in  $K_2 \times (0,\e^{-2}\,  \bar{\nu}]$ and it satisfies
\begin{equation}
\label{trop}
|[w(\cdot,0) \le 2]\cap K_{2}|\le |[w(\cdot,0) \le \e^{-1}  a] \cap K_{2}|\le \nu.
%\qquad -t_{\eps}:= 4\eps^{-2} \bar{t} < \,-\epsilon^{-2} \mu_a.
\end{equation}
We propagate forward in time the information in \eqref{trop} as follows. Fix a time $0<\tau<\nu<1$,
so that we can write down the energy inequality for $(w-2)_{-}$ in the subcylinder $K_2 \times (0,\tau^{2}]$ with $0 \leq \eta \leq 1$ independent of time and such that $\eta=1$ in $K_{1}$, $\eta=0$ outside of $K_{2}$ and $|\partial_{i}\eta|\le C$, to get 
\[
\int_{K_{1}}(w_{t}-2)_{-}^{2}\, dx\le \int_{K_{2}} (w_{0}-2)^{2}_{-}\, dx+C\sum_{i=1}^{N}\int_{0}^{t}\int_{K_{2}}(w_{s}-2)_{-}^{p_{i}}\, dx\, ds,
\] 
for all $t\in (0,\tau^{2}]$.
The second term on the right is bounded by $C\, 2^{N+p_{\rm max}}\, \nu $, while the first one is smaller than $4\, \nu$ due  \eqref{trop}. The term on the left bounds $|[w_{t} \le 1]\cap K_{1}|$, hence we get
\[
|[w_{t} \le 1]\cap K_{1}|\le C\, \nu\, \qquad \forall t\in (0,\tau^{2}],
\]
which implies by integration
\[
|[w\le 1]\cap Q |\le C\, \nu\,  |Q|,\qquad Q:=K_{1}\times (0,\tau^{2}].
\] \noindent 
Let $\tau=2^{-n}$ for some $n\in \N$ to be determined. We partition $K_{1}$ in $2^{Nn}$ dyadic cubes $x_{i}+K_{2^{-n}}=x_{i}+K_{\tau}$ and consider the corresponding  cylinders $\Q_{i}=(x_{i}+K_{\tau})\times(0,\tau^{2}]$. Notice that for any such $\tau$, the latters are intrinsically scaled, since $\Q_{i}=(x_{i}, \nu)+ \Q_{\tau}^{-}(\tau)$. On at least one of these cylinders it must hold
\[
 |[w\le 1]\cap \Q_{i} |\le C\, \nu \, |\Q_{i}|, 
 \]
implying
\[
 |[w\le \tau]\cap \Q_{i} |\le C\, \nu\, |\Q^{-}_{\tau}(\tau)|.
 \] 
 We thus apply Lemma \ref{DG} (see Remark \ref{DGc}), choosing $\nu$ such that $C\, \nu\le \mu_{1}$, (determining $\e$, $\tau$ and $n$ in the process, depending only on $N$ and ${\bf p}$). This implies 
 \[
 w\ge \tau/2\qquad \text{in }\quad (x_{i}, \tau^{2})+T_{\tau,\tau} Q_{1/2}^{-}=(x_{i}, \tau^{2})+\Q_{\tau/2}^{-}(\tau/2)
 \]
 and in particular 
 \[
 w\ge \tau/2\qquad  \text{in}\quad   z_{3}+ \K_{\tau/2}(\tau/2)
 \]
for some $z_{3}$. Scaling back to  $u =\T_{\e r/2,\e r^{-N}/2}^{-1} w$ we get for some $z_{0}\in \Q_{1}^{-}$ the estimate
 \[
 u\ge \tau\, \e\, r^{-N}/4\qquad \text{in} \quad z_{0}+\K_{\tau \e r/4}(\tau\, \e \, r^{-N}/4),
 \]
To conclude the proof of \eqref{infestimate}, it suffices to set
\[
\rho=\frac{\tau\, \e\, r}{4}, \quad \quad  \eps\,  \rho^{-N}=\frac{ \tau\, \e\, r^{-N}}{4},
\]
so that  $\eps= (\tau\, \e/4)^{N+1}$ depends only on $N$ and ${\bf p}$.

 \end{proof} \vskip0.2cm \noindent 
We can now prove the Harnack inequality \eqref{Harnack}.\\

\noindent 
{\em Proof of Theorem \ref{Harnack-Inequality}}.
We begin setting $C_{1}=C$, where the latter is given in Lemma \ref{proposizione-potente}. To define $C_{2}$ and $C_3$, we begin by considering the inequality
\begin{equation}
\label{h1}
u(0, 0)\le C_{3} \inf_{K_{\rho}(M)}u(\cdot, \, M^{2-\bar p}\, (C_{2}\,  \rho)^{\bar p}), \qquad M=u(0, 0)/C_{1}.
\end{equation}
We claim that there exist $\bar D>0$ and functions $\bar A(\cdot)>0$, $\bar B(\cdot)>0$ all depending only on $N$ and ${\bf p}$ such that, whenever
\beq
\label{ABD}
D\ge \bar D, \qquad A\ge \bar A(D), \qquad B\ge \bar B(D), 
\eeq
then it holds
\begin{equation}
\label{k}
\inf_{\K_{r}(M)} u(\cdot, D\, M^{2-\bar p}\,  r^{\bar p})\ge  u(0, 0)/B \quad \text{if}\quad  \K_{A r}(M)\times [-M^{2-\bar p}\, (A\, r)^{\bar p }, M^{2-\bar p}\, (A\, r)^{\bar p}]\subseteq \Omega_{T}.
\end{equation}
Taking $C_{2}\ge \bar D$ and, accordingly,  $C_3\ge \max\{\bar A(C_2), \bar B(C_2)\}$ will then give \eqref{h1} as long as 
\[
\K_{C_3 r}(M)\times [-M^{2-\bar p}\, (C_3\, r)^{\bar p }, M^{2-\bar p}\, (C_3\, r)^{\bar p}]\subseteq \Omega_{T}.
\] 
\medskip
\begin{figure}[htb]
\centering
\begin{tikzpicture}[scale=0.8]
\fill[fill=white!97!black] (-4.6, 6) .. controls (-2.3, 0.4) and (-1.5, -1.3) .. (-1, -1.3) .. controls (-0.5, -1.3) and (0.3, 0.4) .. (2.6, 6);
\draw (-4.6, 6) .. controls (-2.3, 0.4) and (-1.5, -1.3) .. (-1, -1.3) .. controls (-0.5, -1.3) and (0.3, 0.4) .. (2.6, 6);
\fill[fill=white!88!black] (-1, 0) .. controls (-0.5, 0) and (0, 1) .. (2.1, 6) -- (-4.1, 6) .. controls (-2, 1) and (-1.5, 0) .. (-1, 0);
\draw[very thin] (-4.1, 6) .. controls (-2, 1) and (-1.5, 0) .. (-1, 0) .. controls (-0.5, 0) and (0, 1) .. (2.1, 6);

\draw[->] (0, -2.5)--(0, 6.8) node[right]{$t$};
\draw (-5, 0)--(3.5, 0);

\filldraw (-1, 0) circle (1pt);
\draw[very thin, dashed] (-1, 0) node[above]{$\bar x$} -- (-1, -1.3) ;
\draw (-1, 3.5) node{$\displaystyle{\bigcup_{t>0} P_t}$};
\filldraw (0, 0) circle (1pt) node[above right]{$v>C$};

\draw (0, -0.3) node[right]{$\bar t$};
\draw (-0.05, -0.3)-- (0.05, -0.3);

\filldraw (-1, -1.3) circle (1pt);
\draw (0, -1.3) node[right]{$\bar s$};
\draw (-0.05, -1.3)--(0.05, -1.3);

\draw[thick] (-1.85, -0.3)--  node[midway, below]{${\mathcal K}$} (-0.15, -0.3);
\draw[thick] (-2, 6) node[below]{$K_{1}$} --(2, 6) ;

\draw[<->]  (3.8, 6)--(4, 6)-- node[midway, right]{$D$} (4, 0)--(3.8, 0);

\end{tikzpicture}
\caption{Scheme of proof of \eqref{k}. The light-gray part is the support of the Barenblatt starting at $(\bar x, \bar s)$, while ${\cal K}$ is ${\cal K}_{\rho}(\eps\, \rho^{-N})$. }
\label{}

\end{figure}

In order for \eqref{k} to make sense we start by prescribing $\bar A(D)^{\bar p}\ge \max\{D, 1\}$. By assumption, the function $v=\T_{r, M}u$ solves the equation in $\Q_A:=K_{A}\times [-A^{\bar p}, A^{\bar p}]$ and $v(0,0)=C$.  Then \eqref{infestimate} holds, namely there exists $(\bar x, \bar t)\in Q_{1}^{-}$, $\rho\in (0, 1)$ and $\eps=\eps(N, {\bf p})$ such that
\[
\inf_{\bar x+\K_{\rho}(\varepsilon\, \rho^{-N})} v_{\bar t}\ge \varepsilon\,  \rho^{-N}\qquad \text{for}\quad  ( \bar{x}, \bar{t})+\K_{\rho}(\varepsilon\, \rho^{-N})\subseteq K_{1}.
\] 
We choose $\lambda>0$, $-2<s<0$ so that the Barenblatt solution centered at $(\bar x,  s)$ defined as 
\[
b_{\lambda, s}(x, t)=\B_{\lambda}( x-\bar{x}, t-s)
\]
is below $v$ in $K_{A}$, which is implied by 
\[
\begin{cases}
{\rm supp}\, b_{\lambda, s}(\cdot,  \bar{t})\subseteq \bar x+ \K_{\rho}(\varepsilon\, \rho^{-N}),\\[3pt]
\|b_{\lambda, s}(\cdot,  \bar{t})\|_{\infty}\le \varepsilon\, \rho^{-N}.
\end{cases}
\]
By Corollary \ref{corB}, this amounts to 
\[
\begin{cases}
\lambda^{(p_{i}-2)/p_{i}}\, ( \bar{t}-s)^{\alpha_{i}}\le (\varepsilon\, \rho^{-N})^{(p_{i}-\bar p)/p_{i}}\, \rho^{\bar p/ p_{i}}=\varepsilon^{(p_{i}-\bar p)/p_{i}}\, \rho^{\sigma\alpha_{i}}, \\[3pt]
 \lambda\,  ( \bar{t}-s)^{-\alpha}\le \varepsilon\, \rho^{-N},
\end{cases}
\]
which holds true for  $s=\bar s$ obeying $\bar s=\bar{t}-\rho^{\sigma}$ with $\rho<1$ and $\bar \lambda=\lambda(N,{\bf p})$ sufficiently small. Since $\bar{s}>-2$, by Corollary \ref{corB} it holds
\[
b_{\bar \lambda, \bar s}(x, t)\ge\bar \lambda\, \bar \eta\, (t-\bar s)^{-\alpha}\ge\bar \lambda\, \bar \eta \, (t+2)^{-\alpha}  
\]
for all 
\[
 t>0, \qquad x\in \prod_{i=1}^{N}\{ |\bar x_{i}-x_i|< \bar\eta\, \bar \lambda^{(p_{i}-2)/p_{i}}\, (t-\bar s)^{\alpha_i}\}\supseteq P_{t}(\bar x):= \prod_{i=1}^{N}\{ |\bar x_{i}-x_i|< \bar\eta\, \bar \lambda^{(p_{i}-2)/p_{i}}\, t^{\alpha_i}\}.
 \]
We then choose $\bar{\tau}>0$ sufficiently  large so  that $P_{\bar{\tau}}(\bar x)\supseteq K_{1}$ and set $\bar D=\bar \tau$ (this is possible by \eqref{param}, which ensures  $\alpha_i>0$ for each $i=1,\dots ,N$). Then, for any $D\ge \bar D$ we additionally prescribe 
\beq
\label{condbara}
\bar A(D)^{\bar p}\ge D+2 \quad \text{ and}\quad  \bigcup_{\bar x\in K_{1}}{\rm supp}\, {\cal B}_{\bar \lambda}(\cdot-\bar x, D+2)\subseteq K_{\bar A(D)}.
\eeq
 Notice that this choice can be made depending only on the parameters $N, {\bf p}$ and $D$ and that if the latter conditions holds for $\bar A$ then they hold for any $A\ge \bar A$. 
The prescribed conditions on $A$ permits the use of the comparison principle between $v$ and $b_{\bar\lambda, \bar s}$ in $ K_{A}\times [\bar t, D]$ (since on the lateral part of its boundary $b_{\bar \lambda, \bar s}$ vanishes), which then yields
\[
v(\cdot, D)\ge b_{\bar \lambda, \bar s}(\cdot, D)\ge \bar \lambda\, \bar \eta\, (D+2)^{-\alpha}\qquad \text{in $K_{1}$}
\]
for any $D\ge \bar D$. Defining  $\bar B(D)= C(D+2)^{\alpha}/(\bar\eta \bar \lambda)$ and scaling back gives \eqref{k}. 

\medskip

We next deal with the other inequality in \eqref{Harnack},  sketching its proof as some arguments are identical to the previous one (see also \cite{CianiVespri} for a different approach). The constant $C_{1}$ is the same $C$ as before and we claim that the inequality
\begin{equation}
\label{k1}
\sup_{\K_{r}(M)} u(\cdot, -D\, M^{2-\bar p}\, r^{\bar p})\le  B\, u(0, 0)\quad \text{if}\quad  \K_{A r}(M)\times [-M^{2-\bar p}(A\, r)^{\bar p }, M^{2-\bar p}(A\, r)^{\bar p}]\subseteq \Omega_{T}
\end{equation}
(with $M=u(0, 0)/C$) holds true for any $A, B, D$ as in \eqref{ABD}, for a possibly different choice of $ \bar D$ and of the functions $\bar A, \bar B$.  

\begin{figure}[htb]
\centering
\begin{tikzpicture}[scale=0.8]
\fill[fill=white!97!black] (-4.6, 6) .. controls (-2.1, 0.4) and (-1.5, -1.3) .. (-1, -1.3) .. controls (-0.5, -1.3) and (0.1, 0.4) .. (2.6, 6);
\draw (-4.6, 6) .. controls (-2.1, 0.4) and (-1.5, -1.3) .. (-1, -1.3) .. controls (-0.5, -1.3) and (0.1, 0.4) .. (2.6, 6);
\fill[fill=white!88!black] (-1, 0) .. controls (-0.5, 0) and (0, 1) .. (2.1, 6) -- (-4.1, 6) .. controls (-2, 1) and (-1.5, 0) .. (-1, 0);
\draw[very thin] (-4.1, 6) .. controls (-2, 1) and (-1.5, 0) .. (-1, 0) .. controls (-0.5, 0) and (0, 1) .. (2.1, 6);

\draw[->] (2, -2)--(2, 7) node[above right]{$t$};
\draw (-5, 6)--(3, 6);

\filldraw (-1, 0) circle (1pt);
\draw[very thin, dashed] (-1, 6.05) node[above]{$\bar x$}--(-1, 4);
\draw[very thin, dashed] (-1, 3)--(-1, -1.3);
\filldraw (1, 1) circle (1pt) node[below]{$v>C$};
\draw[very thin, dashed] (1, 6) node[above]{$x_{0}$}--(1, 1);

\draw (2, -0.3) node[right]{$\bar t$};
\draw (1.95, -0.3)-- (2.05, -0.3);

\draw (-1, 3.5) node{$\displaystyle{\bigcup_{t>0} P_t}$};

\draw[thick] (-1.8, -0.3)--  node[midway, below]{${\mathcal K}$} (-0.2, -0.3);

\filldraw (-1, -1.3) circle (1pt) (2, 6) circle (1pt) node[above right]{$v=C\, D^{-\gamma}$};
\draw (2, -1.3) node[right]{$\bar s$};
\draw (1.95, -1.3)--(2.05, -1.3);

\draw[<->] (4.1, 6)--(4.3, 6)--node[midway, right]{$D^{1+\gamma(\bar p-2)}$}(4.3, 1)--(4.1, 1);
\end{tikzpicture}
\caption{Scheme of proof of \eqref{k1}. The light-gray part is the support of the Barenblatt starting at $(\bar x, \bar s)$ while ${\cal K}$ is ${\cal K}_{\rho}(\eps\, \rho^{-N})$.}
\label{}

\end{figure}

To prove \eqref{k1}, we fix  $\gamma>N/\bar p$ and start by prescribing 
\[
\bar A(D)^{\bar p}\ge D,\qquad \bar B(D)\ge D^{\gamma}.
\]
Next, consider $A, B, D$ fulfilling \eqref{ABD} together with $\K_{A r}(M)\times [-M^{2-\bar p}(A\, r)^{\bar p }, M^{2-\bar p}(A\, r)^{\bar p}]\subseteq \Omega_{T}$,
but  such that
\beq
\label{assf}
\sup_{\K_{r}(M)} u(\cdot, -D\,M^{2-\bar p}r^{\bar p})>  B\, u(0, 0). 
\eeq
We rewrite the latter in terms of  $v=\T_{r, M D^{\gamma}}u$, which is a solution in ${\cal Q}_{A, D^{-\gamma}}$: the resulting information is
\begin{equation}
\label{ff}
v(0, 0)=C\, D^{-\gamma},\qquad \sup_{\K_{1}(D^{-\gamma})}v(\cdot, -D^{1+\gamma(\bar p-2)})>B\, v(0, 0)\ge  C,
\end{equation}
where we used $B\ge \bar B(D)\ge D^{\gamma}$ in the last inequality. We fix a point $x_{0}\in \K_{1}(D^{-\gamma})$ such that  $v(x_{0}, -D^{1+\gamma(\bar p-2)})>C$
 and suppose that $\bar A(D)$ is additionally large enough so that $v$ is a solution in $(x_{0}, -D^{1+\gamma(\bar p-2)})+Q_{1}$.
We can then apply Lemma \ref{proposizione-potente} and, proceeding exactly as in the first part of the proof, we find
\[
\bar x\in x_{0}+K_{1},\qquad  -D^{1+\gamma(\bar p-2)}-2 \leq \bar s < \bar t \leq -D^{1+\gamma(\bar p-2)}
\]
and $\bar\lambda(N, {\bf p})>0$ such that the Barenblatt solution $b_{\bar\lambda, \bar s}$ centered at $(\bar x, \bar s)$ is below $v$ at the time $\bar t$. As before, for some  $\bar \eta(N,\bf{p}) $ it holds 
\[
b_{\bar\lambda, \bar s}(\cdot, t)\ge  \bar \lambda\,  \bar \eta \, (t+D^{1+\gamma(\bar p-2)}+2)^{-\alpha}\quad \text{in}\quad P_{t+D^{1+\gamma(\bar p-2)}}(\bar x),\quad  \forall t>-D^{1+\gamma(\bar p-2)}.
\]
If needed, we further increase $\bar A(D)$ so that $v$ solves the equation in a rectangle containing the support of any possible $b_{\bar\lambda, \bar s}$ so constructed, up to the time $t= 0$ (through a condition of the type \eqref{condbara}).

\medskip
So far, the definition of the functions $\bar A(D)$ and $\bar B(D)$ is concluded, and we now look for all the values of $D$ such that $0\in P_{\bar D^{1+\gamma(\bar p-2)}}(\bar x)$. Since $x_{0}\in \K_{1}(D^{-\gamma})$ and $\bar x\in x_{0}+K_{1}$, this is true if 
\begin{equation} \label{Dcondition1}
1+D^{-\gamma(p_{i}-\bar p)/p_{i}}\le \bar \eta \, \bar \lambda^{(p_{i}-2)/p_{i}}\, D^{(1+\gamma(\bar p-2))\alpha_{i}},\qquad \forall i=1, \dots, N. 
\end{equation}
We claim that the exponent of $D$ on the left is less than the one on the right. Indeed, from the definition of $\alpha_{i}$, the claim reduces through elementary algebraic manipulations to
\[
\gamma\,  \bar{p}\,  (2-p_i) <N\, (\bar{p}-p_i)+\bar{p},
\] 
which is always true since the left hand side is negative by $p_{i}>2$ and the right hand side is positive by  \eqref{param}.  It follows that \eqref{Dcondition1} holds true for any $D\ge \bar D_{1}$, and in this case we get  by comparison
\beq
\label{aqp}
v(0, 0)\ge b_{\bar\lambda, \bar s}(0, 0)\ge \bar \lambda\,  \bar \eta \, (D^{1+\gamma(\bar p-2)}+2)^{-\alpha}.
\eeq
Next, we claim that there exists $\bar D_{2}$ such that if $D\ge \bar D_{2}$, then
\begin{equation} \label{Dcondition2}
\bar \lambda\,  \bar \eta \, (D^{1+\gamma(\bar p-2)}+2)^{-\alpha}>C\, D^{-\gamma}.
\end{equation}
Indeed, it suffices to show that the exponent on the left is greater than the one on the right, which, recalling that $\alpha=N/(N(\bar p-2)+\bar p)$, amounts to
\[
\gamma-\alpha\, (1+\gamma\, (\bar p-2))=\frac{\gamma\, \bar p-N}{N\, (\bar p-2)+\bar p}>0\quad \Leftrightarrow\quad  \gamma>\frac{N}{\bar p}
\]
as we assumed. Thus \eqref{Dcondition2} is proved, which in turn contradicts the first condition in \eqref{ff} via the lower bound in \eqref{aqp}. All in all, letting $\bar D=\max\{\bar D_{1}, \bar D_{2}\}$ shows that if $A, B, D$ obey \eqref{ABD}, then \eqref{assf} cannot hold, completing the proof of \eqref{k1}. We conclude  choosing the constants $C_2$ and $C_3$ as in the previous step, and finally pick the largest between the so defined constants and previous ones.

\begin{flushright}$\square$\end{flushright}

Finally, we prove the Liouville theorem stated in the Introduction.\\

\noindent
{\em Proof of Corollary \ref{liouville}}. We suppose that $\sup_{\R^N} u > \inf_{\R^N} u$ and let  $\eps \in (0,\sup_{\R^N} u - \inf_{\R^N} u)$. Consider the non-negative solution $v_{\eps}=u-\inf_{\R^{N}} u+\eps/2$ to \eqref{mod2}. By continuity, we can pick a point $x_{\eps}$ such that $v_{\eps}(x_{\eps})=\eps$.  Up to translations, the Harnack inequality \eqref{H} implies that $v_{\eps}\le C_{2}\, \eps$ in $x_{\eps}+K_{\rho}(\eps/C_{1})$, for all $\rho>0$. Letting $\rho\to +\infty$, we get $v_{\eps}\le C_{2}\, \eps$ in the whole $\R^{N}$, i.e.
\[
u\le \inf_{\R^{N}} u+(C_{2}-1/2)\, \eps
\]
in $\R^{N}$ and letting $\eps\to 0$ we get the claim.
\begin{flushright}$\square$\end{flushright}

\end{document}